 \documentclass[final,1p,times]{elsarticle}

\usepackage{amssymb}
 \usepackage{amsthm}
\usepackage{amsmath,amssymb,amsopn,amsfonts,mathrsfs,amsbsy,amscd}
\usepackage{longtable}

\newcommand{\V}{{\mathcal V}}
\newcommand{\pir}{\zeta}

\newcommand{\prs}{\langle\;,\;\rangle}
\newcommand{\prsm}{\langle\;,\;\rangle_{TM}}
\newcommand{\prsa}{\langle\;,\;\rangle_A}

\newcommand{\too}{\longrightarrow}
\newcommand{\om}{\omega}
\newcommand{\esp}{\quad\mbox{and}\quad}

\def\br{[\;,\;]}

\newcommand{\G}{\mathcal{G}}
\newcommand{\g}{\mathfrak{g}}

\newcommand{\Ad}{{\mathrm{Ad}}}
\newcommand{\tr}{{\mathrm{tr}}}
\newcommand{\ric}{{\mathrm{ric}}}

\newcommand{\D}{{\cal D}}

\newcommand{\Om}{\Omega}
\newcommand{\na}{\nabla}

\newcommand{\al}{\alpha}
\newcommand{\be}{\beta}
\newcommand{\ga}{\gamma}
\newcommand{\Ga}{\Gamma}
\newcommand{\e}{\epsilon}

\newcommand{\la}{\lambda}
\newcommand{\De}{\Delta}

\newtheorem{Def}{Definition}[section]
\newtheorem{theo}{Theorem}[section]
\newtheorem{pr}{Proposition}[section]
\newtheorem{Le}{Lemma}[section]
\newtheorem{co}{Corollary}[section]

\newtheorem{remark}{Remark}

\font\bb=msbm10

\def\R{\hbox{\bb R}}

\begin{document}

\begin{frontmatter}


 

\title{ The geometry of generalized Cheeger-Gromoll metrics on the total space of transitive Euclidean  Lie algebroids}

 \author[label1]{Mohamed Boucetta}
 \address[label1]{Universit\'e Cadi-Ayyad\\
 	Facult\'e des sciences et techniques\\
 	BP 549 Marrakech Maroc\\e-mail: m.boucetta@uca.ac.ma}
 \author[label2]{Hasna Essoufi}
 \address[label2]{Universit\'e Cadi-Ayyad\\
 	Facult\'e des sciences et techniques\\
 	BP 549 Marrakech Maroc\\e-mail: essoufi.hasna@gmail.com}
 
 



\begin{abstract} Natural metrics (Sasaki metric, Cheeger-Gromoll metric, Kaluza-Klein metrics etc.. ) on the tangent bundle of a Riemannian manifold is a central topic in Riemannian geometry. Generalized Cheeger-Gromoll metrics is a family of natural metrics  $h_{p,q}$ depending on two parameters with $p\in\R$ and $q\geq0$. This family  has been introduced recently and possesses interesting geometric properties. If $p=q=0$ we recover the Sasaki metric and when $p=q=1$ we recover the classical Cheeger-Gromoll metric. A transitive Euclidean Lie algebroid is a transitive Lie algebroid with an Euclidean product on its total space.
	In this paper, we show that natural metrics can be built in a natural way on the total space of transitive Euclidean Lie algebroids. Then we study the properties of generalized Cheeger-Gromoll metrics on this new context. We show a rigidity result of this metrics which generalizes so far all  rigidity results known in the case of the tangent bundle. We show also that considering natural metrics on the total space of transitive Euclidean Lie algebroids opens new interesting horizons.
	 For instance, Atiyah Lie algebroids constitute an important class of transitive Lie algebroids and we will show that  natural metrics on the total space of Atiyah Euclidean Lie algebroids have interesting properties. In particular, if $M$ is  a Riemannian manifold of dimension $n$, then the Atiyah Lie algebroid  associated to the $\mathrm{O}(n)$-principal bundle of orthonormal frames over $M$ possesses a family depending on a parameter $k>0$ of transitive Euclidean Lie algebroids structures say $AO(M,k)$.    When $M$ is a space form of constant curvature $c$,  we show that there exists two constants $C_n<0$ and $K(n,c)>0$ such that $(AO(M,k),h_{1,1})$ is a Riemannian manifold with positive scalar curvature if and only if $c>C_n$ and $0<k\leq K(n,c)$. 
	
\end{abstract}

\end{frontmatter}








\section{Introduction and main results}\label{section1}

Let $(M,\prsm)$ be a Riemannian manifold of dimension $n$, $\pi_E:E\too M$ a vector bundle of rank $r$ endowed with an Euclidean product $\prs_E$ and $\na ^E$ a linear connection on $E$ which preserves $\prs_E$. Denote by $K:TE\too E$ the connection map of $\na^E$ locally given   by
\[ K\left( \sum_{i=1}^n b_i\partial_{x_i}+\sum_{j=1}^rZ_j\partial_{\mu_j}\right)=
\sum_{l=1}^r\left( Z_l+\sum_{i=1}^n\sum_{j=1}^r b_i\mu_j\Ga_{ij}^l\right)e_l, \]where $(x_1,\ldots,x_n)$ is a system of local coordinates, $(e_1,\ldots,e_r)$ is a basis of local sections of $E$,  $(x_i,\mu_j)$ the associated system of coordinates on $E$ and $\na_{\partial_{x_i}}^E e_j=\sum_{l=1}^r\Ga_{ij}^l e_l$. Then
\[ TE=\ker d\pi_E\oplus \ker K. \]
Let $R^{\na^E}(X,Y)=\na^E_{[X,Y]}-\left(\na^E_X\na^E_Y-\na^E_Y\na^E_X\right)$ be the curvature tensor of $\na^E$.

We define on $E$ a family of Riemannian metrics depending on two parameters $q\geq0$ and $p\in\R$ by putting
\[ h_{p,q}(A,B)=\langle d\pi_E(A),d\pi_E(B) \rangle_{TM}+  \frac1{(1+|a|^2)^p}\left( \langle K(A),K(B) \rangle_{E}+q\langle K(A),a\rangle_E\langle K(B),a\rangle_E\right),\quad A,B\in T_aE.  \]
  These metrics, known as generalized Cheeger-Gromoll metrics,  has been introduced and studied in \cite{benyounes2,benyounes}. The original Cheeger-Gromoll metric \cite{cheeger,musso} corresponds to $p=q=1$ and   $h_{0,0}$ is the Sasaki metric \cite{sasaki}. The main property of these metrics is that $\pi_E:E\too M$ is a Riemannian submersion with totally geodesic fibers and its O'Neill shape tensor is entirely determined by the curvature of $\na^E$. Note that generalized Cheeger-Gromoll metrics constitute a subclass of the class of Kaluza-Klein metrics studied in \cite{wood1,wood2} which is a subclass of $g$-metrics called also natural metrics introduced in \cite{sekizawa1} (see also \cite{abassi1,abassi2}). For a historical review of natural metrics on the tangent bundle of a Riemannian manifold one can see \cite{abassi4}.
  
  This paper has two goals. The first one is to complete the study initiated in \cite{benyounes} by giving new results on the rigidity of these metrics. Indeed, 
 in Section \ref{section2} we will prove the following result.
 \begin{theo} \label{main1}The scalar curvature $s^E$ of $(E,h_{p,q})$ is constant if and only if  the curvature $R^{\na^E}$ of $\na^E$ vanishes, $(p,q)\in\{(0,0),(2,0)  \}$ and the scalar curvature $s^M$ of $M$ is constant. Moreover,
 	\begin{enumerate}
 		\item if $(p,q)=(0,0)$, $R^{\na^E}=0$ and $s^M$ is constant then $s^E=s^M\circ\pi_E,$
 		\item if $(p,q)=(2,0)$, $R^{\na^E}=0$ and $s^M$ is constant then $s^E=s^M\circ\pi_E+4r(r-1)$,
 	\end{enumerate}and  in both cases, $E$ is locally the Riemannian product of $M$ and the fiber. 
 	
 \end{theo}
 When $E=TM$, $\prsm=\prs_E$ and $\na^E$ is the Levi-Civita connection of $\prsm$ we get the following result which is new and completes the results obtained in  \cite{benyounes}.
 
 \begin{co} \label{main1c}The scalar curvature $s^{TM}$ of $(TM,h_{p,q})$ is constant if and only if one of the following holds:
 	\begin{enumerate}
 		\item  $(p,q)=(0,0)$, $R^M=0$ and in this case $(TM,h_{p,q})$ is flat.
 		
 		\item  $(p,q)=(2,0)$, $R^M=0$ and $s^{TM}=4n(n-1)>0$.
 	\end{enumerate}
 	
 \end{co}
 
 Note that all the classical rigidity results of the Sasaki metric can be derived from Corollary \ref{main1c}. On the other hand,  $R^{\na^E}=0$ if and only if the O'Neill shape tensor of the Riemannian submersion $\pi:(E,h_{p,q})\too (M,\prsm)$ vanishes which is equivalent to $E$ being locally the Riemannian product of $M$ and the fiber. So,  Theorem \ref{main1} can be stated a follows: $(E,h_{p,q})$ has constant scalar curvature if and only if $E$ is locally the Riemannian product of $M$ and the fiber, $M$ has constant scalar curvature and the fiber has constant scalar curvature. We will show in Lemma \ref{le1} that the restriction of $h_{p,q}$ to a fiber has constant scalar curvature if and only if $(p,q)\in\{(0,0),(2,0)  \}$. In conclusion, Theorem \ref{main1} is a strong rigidity result since it cuts all hope of building interesting examples of locally symmetric spaces, Einstein manifolds and so on, by using generalized Cheeger-Gromoll metrics on $E$.

  To our knowledge, even if one can define natural metrics in the general sitting of an Euclidean bundle over a Riemannian manifold, only the case of the tangent bundle has been considered so far except in \cite{benyounes2} where harmonic sections of Euclidean bundles have been considered. The reason is the difficulty of finding interesting examples. It is easy to build an Euclidean bundle $E\too M$ over a Riemannian manifold but it is more difficult to find a connexion on $E$ which preserves the Euclidean product and it is far more difficult to find one which has some link to the geometry of the Riemannian manifold $M$. Our second goal in this paper is to remedy  this situation and introduce a large class of Euclidean bundles where natural metrics can be defined and have interesting properties.
   Indeed, the first author has introduced Riemannian Lie algebroids in \cite{boucetta} (we use in this paper the terminology Euclidean instead of Riemannian)  and has  shown that the analogous of Sasaki metric  can be build on the total space $A$ of a Riemannian transitive Lie algebroid. More precisely, it has been shown that $TA$ splits into a vertical part and a horizontal one and it is what one needs to build natural metrics. The construction of this splitting in \cite{boucetta} is based on the properties of connections in the context of Lie algebroids. But, when we started studying generalized Gromoll-Cheeger metrics on transitive Riemannian Lie algebroids, we noticed that they constitute a particular case 
  of generalized Gromoll-Cheeger metrics on Euclidean vector bundles  introduced above.
 Let us give more details on this now. 
 
 Let $\pi_A:A\too M$ be a vector bundle endowed with an Euclidean product $\prsa$. Suppose that $A$ carries a structure of transitive Lie algebroid, i.e.,  a surjective bundle homomorphism $\rho:A\too TM$, a structure of real Lie algebra  $\br_A$ on $\Ga(A)$ such that
 \[ [a,fb]_A=f[a,b]_A+\rho(a)(f)b,\quad a,b\in\Ga(A),f\in C^\infty(M). \]
 We call $(A,\prsa,\rho,[\;,\;]_A)$ a transitive Euclidean Lie algebroid. If $\G=\ker\rho$ then $\pi_\G:\G\too M$ is a  Lie algebroid with vanishing anchor called the adjoint Lie algebroid of $A$ and we have an exact sequence of Lie algebroids called {\it Atiyah sequence}
 \begin{equation}\label{eq20}0\too\G\too A\stackrel{\rho}\too TM\too0.\end{equation}
 There are two important objects naturally associated to  $(A,M,\rho,\prsa)$.
 \begin{enumerate}
 	\item The analogous of the Levi-Civita
 	connection. Indeed,  the Koszul formula
 	\begin{eqnarray}
 	2\langle\D_ab,c\rangle_A&=&{\rho}(a).\langle b,c\rangle_A+{\rho}(b).\langle a,c\rangle_A-
 	{\rho}(c).\langle a,b\rangle_A\label{koszulb}\\
 	&&	+\langle[c,a]_A,b\rangle_A+\langle[c,b]_A,a\rangle_A+\langle[a,b]_A,c\rangle_A,\quad a,b,c\in\Ga(A)\nonumber\end{eqnarray}defines
 	a linear $A$-connection  which is characterized by the fact that
 	$\D$ is metric, i.e., $\rho(a).\langle b,c\rangle_A=\langle\D_ab,c\rangle_A+\langle b,\D_ac\rangle_A$
 	and $\D$ is torsion free, i.e., $\D_ab-\D_ba=[a,b]_A.$\\
 	The connection $\D$ is well-known as the \emph{ Levi-Civita $A$-connection} associated to
 	the Riemannian metric $\prsa$. The reader can consult \cite{boucetta, fernandes} for a detailed study of connections on Lie algebroids.
 	
 	\item A splitting of the Atiyah sequence of $A$.  Indeed, For any $x\in M$, we denote by $\G_x^\perp$ the orthogonal of $\G_x$ with respect to $\prsa$ thus $$A=\G\oplus\G^\perp.$$ The restriction of 
 	$\rho$ to $\G^{\perp}$ is an isomorphism onto $TM$ and its inverse $\ga:TM\too \G^\perp$ defines a splitting of the Atiyah sequence.
 \end{enumerate}

 From these two objects one can extract  the  necessary ingredients for defining natural metrics and, in particular, generalized Cheeger-Gromoll metrics on $A$. Indeed,  we have a Riemannian metric on $M$ and a connection $\na^A$ on $A$ given by
 \[ \langle X,Y\rangle_{TM}=\langle \ga(X),\ga(Y)\rangle_A\esp \na^A_Xa=\D_{\ga(X)}a,\quad  X,Y\in\Ga(TM), a\in\Ga(A),  \]and since $\D$ is metric, $\na^A$ preserves $\prsa$.  The curvature of $\na^A$ plays an important role in the study of the geometry of generalized Cheeger-Gromoll metrics on $A$. It depends on the Lie algebroid structure and on the metric $\prsa$.  So we call it  {\it  principal curvature} of the transitive Euclidean Lie algebroid $A$.

 There are many reasons why generalized Cheeger-Gromoll metrics on transitive Euclidean Lie algebroids are interesting:
 \begin{enumerate}
 	\item They generalize naturally generalized Cheeger-Gromoll metrics on the tangent space of a Riemannian manifold. The tangent space of a Riemannian manifold has a natural structure of transitive Euclidean Lie algebroid.
 	\item  When a transitive Euclidean Lie algebroid $A$ is endowed with a generalized Cheeger-Gromoll metric $h_{p,q}$, the O'Neill shape tensor of the Riemannian submersion $\pi_A:(A,h_{p,q})\too (M,\prsm)$ is encoded in the  principal curvature of $A$ and can be computed explicitly (see Proposition \ref{fb}). It involves the curvature of $M$ and on the Lie algebroid structure. So the geometry of $(A,h_{p,q})$ is deeply linked to the geometry of $(M,\prsm)$ and the Lie algebroid structure as one can see in Proposition \ref{fb1} where we show that the vanishing of the principal curvature has drastic consequences on $(M,\prsm)$ and the Lie algebroid.
 	\item There is a large class of transitive Lie algebroids, namely, Atiyah Lie algebroids associated to principal bundles (see \cite{kubarski}). Euclidean Atiyah Lie algebroid turn out to be interesting and we devote Section \ref{section4} to give a precise description of them.
	\item To any Riemannian manifold $(M,\prsm)$ we can associate canonically a transitive Lie algebroid.  Indeed, the $O(n)$-principal bundle of orthonormal frames over $M$   has an associated Atiyah Lie algebroid.  We will show in Section \ref{section4} that this Lie algebroid can be identified with $TM\oplus\mathrm{so}(TM)$ where $\mathrm{so}(TM)=\bigcup_{x\in M}\mathrm{so}(T_xM)$ and $\mathrm{so}(T_xM)$ is the Lie algebra of skew-symmetric endomorphisms of $T_xM$.  The Lie bracket $\br_A$ and the anchor $\rho$  are given by
	\begin{eqnarray*} [X+F,Y+G]_A&=&[X,Y]+\left\{ \na_X^M(G)-\na_Y^M(F)+[F,G]-R^M(X,Y)\right\},\;\rho(X+F)=X,\\&&\quad X,Y\in\Ga(TM), F,G\in\Ga(\mathrm{so}(TM)) ,  \end{eqnarray*}where $\na^M$ is the Levi-Civita connection of $M$ and $R^M$ is its curvature. Moreover, this Lie algebroid can be endowed with a family of Euclidean products $\prs_k$ given by
	\[ \langle X+F,Y+G\rangle_k= 
	\langle X,Y\rangle_{TM}       -k\tr(F\circ G). \]
	We denote by $AO(M,k)$ the Lie algebroid $TM\oplus\mathrm{so}(TM)$ endowed with the Euclidean product $\prs_k$.
	
	Our second main result can be compared to the main result obtained in \cite{gud}. We recall this result in order to give  the reader a possibility of comparing it to ours.
	
 \begin{theo}[\cite{gud}]\label{gud} Let $(M,\prsm)$ be a Riemannian manifold of dimension $n$ and of constant sectional curvature $c$. Then
 	\begin{enumerate}
 		\item If $n=2$ then there exists a constant $C_2\geq40$ such that $(TM,h_{1,1})$ has positive scalar curvature if and only if $c\in(0,C_2)$.
 		\item If $n>2$ then there exists two constants $C_n\geq60$ and $c_n<0$ such that $(TM,h_{1,1})$ has positive scalar curvature if and only if $c\in(c_n,C_n)$.
 	\end{enumerate}
 	
 \end{theo}

  We can state now our second main result and one can see that the conditions on the curvature in our result are far less restrictive than those in Theorem \ref{gud}.
 
 \begin{theo}\label{main5} Let $(M,\prsm)$ be a Riemannian manifold of dimension $n$ with constant sectional curvature $c$. Then
 	\begin{enumerate}
 		\item[$(i)$] If $c=0$ then for any $k>0$, $(AO(M,k),h_{1,1})$ has positive scalar curvature.
 		\item[$(ii)$] If $c\not=0$ and $n=2$ then $(AO(M,k),h_{1,1})$ has positive scalar curvature if and only if $c>2(1-\sqrt{2})\simeq-0,82$ and $0<k\leq \frac{2(c+2\sqrt{1+c})}{c^2}$.
 		\item[$(iii)$] If $c\not=0$ and $n\geq3$ then $(AO(M,k),h_{1,1})$ has positive scalar curvature if and only if
 		\[ c>\frac{2(a-\sqrt{a^2+bd})}{d}=C_n\esp 0<k\leq \frac{2(cd+2\sqrt{d}\sqrt{b+ac})}{dc^2}=K(n,c), \]
 		where $a=n(n-1)$, $b=(r-1)(r-2)$,  $d=4(n-2)$ and $r=\frac{n(n+1)}2$.
 		Moreover, $C_n<0$, $K(n,c)>0$ and, for instance,
 		\[ C_3\simeq -2,3,  C_4\simeq -3,7, C_5\simeq -5,1, C_6\simeq-6,6, C_{20}\simeq -39,7.\]
 	\end{enumerate}
 	
 \end{theo}
 
\end{enumerate}
 Finally, this work opens new horizons, namely, it gives the basis of further study of all kind of natural metrics (studied on the tangent bundle of a Riemannian manifolds) on the total spaces of  transitive Euclidean Lie algebroids.
 
 The paper is organized as follows. In Section \ref{section2}, we give the main properties of generalized Cheeger-Gromoll metrics on the total space of an Euclidean vector bundle,  we prove Theorem \ref{main1} and we derive some of its corollaries.  In Section \ref{section3}, we give a complete description of transitive Euclidean Lie algebroids (see Theorem \ref{main6}),  we compute their principal curvature and we give the geometrical consequences of its vanishing. Section \ref{section4} is devoted to the characterization of Atiyah Euclidean Lie algebroids (see Corollary \ref{atiyah}).  In Section \ref{section5}, we prove Theorem \ref{main5}.

 \section{  Generalized Cheeger-Gromoll metrics on the total space of Euclidean vector bundles and their  rigidity }\label{section2}
 
  Generalized Cheeger-Gromoll metrics on the tangent space of a Riemannian manifold  were introduced and studied in \cite{benyounes}. In this section, we consider a more general sitting, namely, generalized Cheeger-Gromoll metrics on the total space of Euclidean vector bundles over a Riemannian manifold. We will show that these metrics are rigid 
    recovering some classical results and establishing other ones  which are new even in the classical case of generalized Cheeger-Gromoll metrics on the tangent bundle.

 \subsection{Definitions and immediate properties}\label{subsection21} 
  
  Let $(M,\prsm)$ be a $n$-dimensional Riemannian manifold and $\pi_E:E\too M$ a vector bundle  of rank $r$ endowed with an Euclidean product $\prs_E$. 
  We suppose that there exists a linear connection $\na^E$ on $E$ for which $\prs_E$ is parallel. We denote by $(x,a)$ an element of $E_x$. For any $(x,a)\in E$  there exists an injective linear map $h^{(x,a)}:T_xM\too T_{(x,a)}E$ given in a coordinates system $(x_i,\be_j)$ associated to a local trivialization $(s_1,\ldots,s_r)$ of $E$ around $x$ by
  \begin{equation}\label{h} h^{(x,a)}(x,u)=\sum_{i=1}^nu_i\partial_{x_i}-\sum_{k=1}^r\left(\sum_{i=1}^n\sum_{j=1}^r u_i\be_j\Ga_{ij}^k\right)\partial_{\be_k},\end{equation}where
  $$u=\sum_{i=1}^nu_i\partial_{x_i},\; \na^E_{\partial_{x_i}}s_j=\sum_{k=1}^r\Ga_{ij}^ks_k\esp a=\sum_{i=1}^r\be_is_i.$$ Moreover,  if $ \mathcal{H}_{(x,a)}E$ denotes the image of $h^{(x,a)}$ then   
  \begin{equation}\label{eq21} TE=\V E\oplus \mathcal{H} E,\end{equation}where $\V E=\ker d\pi_E$.
  
  For any $\al\in\Ga(E)$ and for any $X\in\Ga(TM)$, we denote by $\al^v\in\Ga(TE)$ and $X^h\in\Ga(TE)$ the  vertical and horizontal vector field associated to $\al$ and $X$. The flow of $\al^v$ is given by $\Phi^\al(t,(x,a))=a+t\al(x)$ and $X^h$ is given by $X^h(x,a)=h^{(x,a)}(X(x))$.  To prove the following proposition one can mimic the well-known proof in the case where $E=TM$, $\prs_E=\prsm$ and $\na^E$ is the Levi-Civita connection of $\prsm$.
  
  \begin{pr}\label{prb} For any $X,Y\in \Ga(TM),\al,\be\in\Ga(E)$,
  	\[ [\al^v,\be^v]=0,\;[X^h,\al^v]=(\na^E_{X}\al)^v\esp [X^h,Y^h]((x,a))=[X,Y]^h((x,a))+(R^{\na^E}(X,Y)a)^v, \]where $R^{\na^E}$ is the curvature of ${\na^E}$ given by $R^{\na^E}(X,Y)=\na^E_{[X,Y]}-\left(\na^E_X\na^E_Y-\na^E_Y\na^E_X\right)$.
  	
  \end{pr}
  
  The generalized Cheeger-Gromoll metrics is a family of Riemannian metrics on $E$ depending on two parameters $p\in\R$ and $q\geq0$ and given by
  \begin{eqnarray}
  	h_{p,q}(X^h,Y^h)&=& \langle X,Y\rangle_{TM}\circ\pi_E  ,\; h_{p,q}(X^h,\al^v)=0,\; \al,\be\in\Ga(E),\; X,Y\in\Ga(TM),\nonumber\\ 
  	h_{p,q}(\al^v,\be^v)((x,a))&=&(1+|a|^2)^{-p}\left(\langle\al,\be\rangle_E+q\langle\al,a\rangle_E\langle\be,a\rangle_E\right).\label{cg}
  \end{eqnarray}
  We will denote by $\om_q(a)=\frac1{1+q|a|^2}$ with $\om_1=\om$.
  
  Note that $h_{0,0}$ is the Sasaki metric, $h_{1,1}$ is the classical Cheeger-Gromoll metric and $h_{2,0}$ is the stereographic metric.
  
  To compute the Riemannian invariants of $(E,h_{p,q})$ (Levi-Civita connection and the different curvatures), we will use the following facts:
  \begin{enumerate}
  	\item[$(i)$] The projection $\pi_E:(E,h_{p,q})\too (M,\prsm)$ is a Riemannian submersion with totally geodesic fibers and hence the different Riemannian invariants can be computed by using O'Neill formulas (see \cite[chap. 9]{bes}). Here the  O'Neill shape tensor, say $B$, is given by the expression of $[X^h,Y^h]$. So, by virtue of Proposition \ref{prb}, we get
  	\begin{equation}\label{b}
  	B_{X^h}Y^h((x,a))=\frac12\V[X^h,Y^h]=\frac12(R^{\na^E}(X,Y)a)^v.
  	\end{equation}
  	\item[$(ii)$] O'Neill's formulas involve the Riemannian invariants of $(M,\prsm)$, the tensor $B$ and the Riemannian invariants of the restriction of $h_{p,q}$ to the fibers. The latest have been computed in \cite{benyounes} and we will use them.
  	
  \end{enumerate}
  Based on these facts, the Levi-Civita connection  $\bar{\na}$ of $(E,h_{p,q})$ is given by
  \begin{eqnarray*}
  	\bar{\na}_{X^h}Y^h&=&(\na_X^MY)^h+B_{X^h}Y^h,\;\bar{\na}_{X^h}\al^v=(\na^E_{X}\al)^v+B_{X^h}\al^v,\;
  	\bar{\na}_{\al^v}X^h=B_{X^h}\al^v,\\
  	(\bar{\na}_{\al^v}\be^v)(a)&=&-p\om(a)[\langle \al,a\rangle_E \be+\langle \be,a\rangle_E \al]^v+(p\om(a)+q)\om_q(a)\langle\al,\be\rangle_E U(a)\\&&+pq\om(a)\om_q(a) 
  	\langle \al,a\rangle_E\langle \be,a\rangle_E U(a),\\
  	h_{p,q}(B_{X^h}\al^v,Y^h)&=&-h_{p,q}(B_{X^h}Y^h,\al^v),\; U(a)=a^v.
  \end{eqnarray*}The expression of $\bar{\na}_{\al^v}\be^v$ has been computed in \cite[Proposition 2.2]{benyounes}. 
  
  \begin{remark} \label{rem1} If $B=0$, from the relations above we can see that both $\V E$ and, $\mathcal{H}E$ are  parallel and according to de Rham's holonomy theorem, $(E,h_{p,q})$ is, at least locally, the Riemannian product  $M\times E_x$. Thus  $R^{\na^E}=0$ if and only if $(E,h_{p,q})$ is locally the Riemannian product of $M$ and the fiber.
  	
  \end{remark}

 The following formulas we will use later were established in \cite[Propositions 2.4, 2.9 and 2.11]{benyounes}.  
  
  \begin{pr}\label{f} We denote by $K^v$, $\ric^v$ and $s^v$ respectively the sectional curvature, the Ricci curvature and the scalar curvature of the restriction of $h_{p,q}$ to the fibers of $E$. Put
  	\[ F=p\om\om_q((p+2q-2)\om-q)\esp G=(p^2\om^2-p(p-2)\om+q )\om_q. \]	
  	Then:
  	\begin{enumerate}
  		\item[$(i)$] if $\al,\be\in E$ such that $\langle\al,\al\rangle_E=\langle\be,\be\rangle_E=1$ and $\langle\al,\be\rangle_E=0$ then
  		\[ K^v(\al^v,\be^v)=\frac{\om^{-p}}{1+q(\langle\al,a\rangle_E^2+\langle\be,a\rangle_E^2)}(F(a)(\langle\al,a\rangle_E^2+\langle\be,a\rangle_E^2)+G(a)). \]
  		\item[$(ii)$] For any $\al,\be\in\Ga(E)$,
  		\[ \ric(\al^v,\be^v)(a)=(|a|^2\om_q(a)F(a)+(r-2+\om_q(a))G(a))\langle\al,\be\rangle_E+
  		((r-1-\om_q(a))F(a)+q\om_q(a)G(a))
  		\langle\al,a\rangle_E\langle\be,a\rangle_E. \]		 
  		\item[$(iii)$] $s^v(a)=f(|a|^2)$ where
  		\begin{eqnarray*}
  			f(t)&=&\frac{(r-1)(1+t)^p}{(1+qt)^2(1+t)^2}\left( e t^3+bt^2+ct+d\right),\\
  			e&=&q^2(r-2), b=q((2 -r) p^{2}  + 2 \, (r-3) p   + 2 \, (r-2) q + r)   ,\\
  			c&=&(2-r) p^{2} + 2 \, (r-1) p q + (r -2)q^{2} + 2 \, (r-2) p  + 2 \, r q   ,\; d=r(2p+q).
  		\end{eqnarray*}	  	\end{enumerate}
  	\end{pr}
  	
  	\begin{proof} The expressions of the sectional curvature and the Ricci curvature are given in \cite[Proposition 2.4 and Proposition 2.9]{benyounes}. For the scalar curvature we have expanded the expression given in \cite[ Proposition 2.11]{benyounes}. Indeed, the scalar curvature $s^v$ is given in \cite[ Proposition 2.11]{benyounes} by
  		\[ s^v=(r-1)\om^{-p}(2\al-(r-2)G)\esp \al=|a|^2\om_qF+(r-2+\om_q)G. \]
  		When we expand this expression we get the desired formula.
  		\end{proof}
  	
  		\subsection{Rigidity  of Cheeger-Gromoll metrics on the total space of  Euclidean vector bundles}\label{subsection22}

  		In this section, we give a precise image of what one can expect  from generalized Cheeger-Gromoll metrics  on the total space of an Euclidean vector bundle in term of constance of different curvatures (scalar,  Ricci  or sectional curvature) or local symmetry. 
  		
  		Through this subsection $\pi_E:E\too M$ is an Euclidean vector bundle over a Riemannian manifold and $\na^E$ a linear connection on $E$ which preserves $\prs_E$.
  		
  		 Corollary 2.5 in \cite{benyounes} asserts that when $E=TM$ the only generalized Cheeger-Gromoll metric with flat fibers is the Sasaki metric $h_{0,0}$. The following lemma gives a far more accurate assertion.
  		
  		\begin{Le}\label{le1} The scalar curvature $s^v$ of the restriction of $h_{p,q}$ to a fiber $E_x$ is constant if and only if $(p,q)=(0,0)$ or $(p,q)=(2,0)$. If $(p,q)=(0,0)$ then the fibers are flat and if $(p,q)=(2,0)$ then they have constant scalar curvature $4r(r-1)$.
  			
  		\end{Le}
  		\begin{proof} According to Proposition \ref{f}, we have $s^v(a)=f(|a|^2)$.
  			A direct computation using the software Sage gives  
  			\[ f'(t)=\frac{(r-1)(1+t)^p}{(1+qt)^3(1+t)^3}\left( a_1t^4+b_1t^3+c_1t^2+d_1t+e\right) \]
  			where
  			\begin{eqnarray*}
  				a_1&=&{\left(r - 2\right)} {\left(p - 1\right)} q^{3},\;\\
  				b_1&=&-{\left(r p^{3} - 4 \, r p^{2} - 2 \, p^{3} - 2 \, r p q + 2 \, r p + 10 \, p^{2} + 3 \, r q + 4 \, p q + r - 10 \, p - 6 \, q + 2\right)} q^{2},\\
  				c_1&=&-2 \, r p^{3} q + 2 \, r p^{2} q^{2} + r p q^{3} + 7 \, r p^{2} q + 4 \, p^{3} q - 2 \, r p q^{2} - 2 \, p^{2} q^{2} - 3 \, r q^{3} - 2 \, p q^{3} - 5 \, r p q - 16 \, p^{2} q \\&&- 3 \, r q^{2} + 2 \, p q^{2} + 6 \, q^{3} + 12 \, p q - 6 \, q^{2},\\
  				d_1&=&-r p^{3} + 3 \, r p^{2} q - r q^{3} + 3 \, r p^{2} + 2 \, p^{3} - 6 \, r p q - 3 \, r q^{2} + 2 \, q^{3} - 2 \, r p - 6 \, p^{2} - 6 \, p q - 6 \, q^{2} + 4 \, p,\\
  				e_1&=&{\left(p^{2} - p q - q^{2} - 2 \, p\right)} {\left(r + 2\right)}.
  			\end{eqnarray*}
  			So $s^v$ is constant if and only if $a_1=b_1=c_1=d_1=e_1=0$. Or $a_1=0$ iff $q=0$ or $p=1$. If $p=1$ we replace in $e_1$ and we get $q^2+q+1=0$ which is impossible. So $q=0$ and $p\not=1$. Then $a_1=b_1=c_1=0$,  $d_1=(2-r)p(p-2)(p-1)$ and $e_1=(r+2)p(p-2)$ and the result follows. The last statement is a consequence of Proposition \ref{f} $(i)$.
  		\end{proof}
  		
  		We can give now a proof of Theorem \ref{main1}.

  		\begin{proof} According the O'Neill formulas (see \cite[pp.244]{bes}) and the expression of the scalar curvature of the restriction of $h_{p,q}$ to the fibers given in Proposition \ref{f}, we have
  			\[ s^E(x,a)=s^M(x)+f(|a|^2)-\sum_{i,j}h_{p,q}(B_{X_i^h}X_j^h,B_{X_i^h}X_j^h), \]
  			where
  			\begin{eqnarray*}
  				f(t)&=&\frac{(r-1)(1+t)^p}{(1+qt)^2(1+t)^2}\left( q^2(r-2)t^3+bt^2+ct+d\right),\\
  				 b&=&q((2 -r) p^{2}  + 2 \, (r-3) p   + 2 \, (r-2) q + r)   ,\\
  				c&=&(2-r) p^{2} + 2 \, (r-1) p q + (r -2)q^{2} + 2 \, (r-2) p  + 2 \, r q   ,\; d=r(2p+q).
  			\end{eqnarray*}	
  			and  $(X_1,\ldots,X_n)$ is a local frame of orthonormal vector field on $M$. Now
  			\[ h_{p,q}(B_{X_i^h}X_j^h,B_{X_i^h}X_j^h)=\frac{1}{4(1+|a|^2)^p} \langle R^{\na^E}(X_i,X_j)a,R^{\na^E}(X_i,X_j)a\rangle_E. \]
  			Note that we have used here the fact that $\langle R^{\na^E}(X_i,X_j)a,a\rangle_E=0$ which is a consequence of the fact that $\na^E$ preserves $\prs_E$. We deduce that
  			 \[ s^E(x,a)=s^M(x)+f(|a|^2)-\frac{\xi(a,a)}{4(1+|a|^2)^p}, \]where
  			 $\xi$ is the symmetric 2-form on $E$ given by
  			 \[ \xi(a,b)=\sum_{i,j}\langle R^{\na^E}(X_i,X_j)a,R^{\na^E}(X_i,X_j)b\rangle_E,\quad a,b\in\Ga(E). \]

  			If $($$(p,q)=(0,0)$, $R^{\na^E}=0$ and $s^M$ is constant$)$ or $($$(p,q)=(2,0)$, $R^{\na^E}=0$ and $s^M$ is constant$)$ then, according to Lemma \ref{le1}, $f$ is constant and hence $s^E$ is constant.
  			
  			Suppose now that $s^E$ is constant. Since $\xi(0,0)=0$ and $f(0)=r(r-1)(2p+q)$ we get 
  			$$s^E-s^M\circ\pi_A=r(r-1)(2p+q)$$ and hence $s^M$ is constant.
  			
  			If $R^{\na^E}=0$ then $s^v$ is constant and,  according to Lemma \ref{le1}, $(p,q)=(0,0)$ or $(p,q)=(2,0)$.
  			
  			If $(p,q)=(0,0)$ or $(p,q)=(2,0)$ then $f$ is constant and then $|B|^2$ is constant and since $\xi(0,0)=0$ then $R^{\na^E}=0$.
  			
  			Suppose  now $R^{\na^E}\not=0$, $(p,q)\not=(0,0)$ and $(p,q)\not=(2,0)$. So $\xi\not=0$ and we can choose $a$ such that $\xi(a,a)\not=0$ and $|a|=1$.  For any $t\in\R$,
  			\[ s^E(ta)-s^M\circ \pi_A(ta)=f(t^2)-\frac{\xi(a,a)t^2}{4(1+t^2)^p}=r(r-1)(2p+q). \]
  			Thus
  			\[ f(t)-\frac{\xi(a,a)t}{4(1+t)^p}=r(r-1)(2p+q), \quad \xi(a,a)>0, t\geq0.\eqno(E) \]
  			Suppose that $p>1$. Then ($p>1$ and $p\not=2$) or ($p=2$ and $q\not=0$) and hence
  			\[ \lim_{t\too\infty}\frac{t}{(1+t)^p}=0\esp \lim_{t\too\infty}f(t)=+\infty, \]which is impossible by virtue of $(E)$. 
  			
  			Suppose now that  $p<1$.  Then ($p<1$ and $p\not=0$) or ($p=0$ and $q\not=0$) and hence
  			\[ \lim_{t\too\infty}\frac{t}{(1+t)^p}=+\infty\esp \lim_{t\too\infty}f(t)=0, \]which is impossible by virtue of $(E)$.
  			
  			Let finish by showing that  the case $p=1$ is also impossible. Indeed, if $p=1$ then, by taking the derivative of $(E)$, we get
  			\[ \frac{(r-1)}{(1+qt)^3(1+t)^2}\left( b_1t^3+c_1t^2+d_1t+e_1\right)-\xi(a,a)\frac{1}{(1+t)^{2}}=0, \]
  			where
  			\begin{eqnarray*}
  				b_1&=&-{\left(r - 2\right)} q^{3},
  				c_1=-{\left(2 \, r q + 3 \, r - 4 \, q + 6\right)} q^{2},
  				d_1=-{\left((r-2) q^{2} + 3  (r+2) q  + 3  r  + 6\right)} q,\\
  				e_1&=&-{\left(q^{2} + q + 1\right)} {\left(r + 2\right)}.
  			\end{eqnarray*}For $t=0$ we get $\xi(a,a)=(r-1)e_1<0$ which it is impossible. This achieves to prove the first part of the theorem. 
  			
  			On the other hand, if $($$(p,q)=(0,0)$, $R^{\na^E}=0$ and $s^M$ is constant$)$ then $s^E=s^M\circ\pi_A$.
  			On the other hand, if $($$(p,q)=(2,0)$, $R^{\na^E}=0$ and $s^M$ is constant$)$ then $s^E=s^M\circ\pi_E+4r(r-1).$
  		\end{proof}
  		
  		According to Remark \ref{rem1}, $R^{\na^E}=0$ if and only if $(E,h_{p,q})$ is locally the Riemannian product of $M$ with the fiber. From this and Theorem \ref{main1}, one can easily deduce  the following corollaries.

  		\begin{co}\label{main2} $(E,h_{p,q})$ is an Einstein manifold  with Einstein constant $\la$ if and only if one of the following situations occurs:
  			\begin{enumerate}
  				\item $(p,q)=(0,0)$, $R^{\na^E}=0$, $\la=0$ and $(M,\prsm)$ is Ricci flat;
  				\item $(p,q)=(2,0)$, $R^{\na^E}=0$, $\la=4(r-1)$ and $(M,\prsm)$ is an Einstein manifold with Einstein constant $4(r-1)$.
  			\end{enumerate}
  			
  		\end{co}

  		When $E=TM$, $\prsm=\prs_E$ and $\na^E$ is the Levi-Civita connection of $\prsm$ we get the following  result which precises Corollary 2.10 in \cite{benyounes}.
  		
  		\begin{co}\label{main2c} $(TM,h_{p,q})$ is an Einstein manifold   if and only if $(p,q)=(0,0)$ and $R^M=0$ and in this case $(TM,h_{p,q})$ is flat.
  		\end{co}

  		\begin{co}\label{main3} $(E,h_{p,q})$ is locally symmetric   if and only if $R^{\na^E}=0$,  $(p,q)\in\{(0,0),(2,0)  \}$ and $M$ is locally symmetric.
  		\end{co}

  		\begin{co}\label{main3c} $(TM,h_{p,q})$ is locally symmetric   if and only if 
  			$(p,q)\in\{(0,0),(2,0)  \}$ and $R^M=0$.
  		\end{co}	
  		The following corollary is a consequence of O'Neill formulas and Theorem \ref{main1}.
  		
  		\begin{co}\label{main4} The following assertions are equivalent:
  			\begin{enumerate}
  				\item $(E,h_{p,q})$ has  constant sectional curvature.
  				\item The curvature of $(E,h_{p,q})$ vanishes.
  				\item $p=q=0$, $R^M=0$ and $R^{\na^E}=0$.
  			\end{enumerate}

  		\end{co}

\section{ Transitive Euclidean Lie algebroids and their principal curvature}\label{section3}

In this section, we consider a transitive Lie algebroid $(A,M,\rho)$ and an Euclidean product $\prsa$ on $A$. We will show that we can define canonically a Riemannian metric $\prsm$ on $M$ and a connection $\na^A$ on $A$ which preserves $\prsa$. Hence, according to the last section, we can define the family of generalized Cheeger-Gromoll metrics on $A$. The O'Neill shape tensor of these metrics is given by the curvature of $\na^A$. We compute this curvature to discover that it depends on the Lie algebroid structure and the curvature of $M$ and hence encompasses a rich geometrical situation.

A Lie algebroid  over a smooth manifold $M$ is a vector bundle
$\pi_A:A\too M$ together with a $\R$-Lie algebra structure $\br_A$ on  $\Ga(A)$ and a vector bundle homomorphism  $\rho:A\too TM$ called
{\it anchor} such that,
for any  $a,b\in\Ga(A)$ and for any $f\in C^\infty(M)$, we have the Leibniz identity
\begin{equation}\label{eq1}[a,fb]_A=f[a,b]_A+\rho(a)(f)b.\end{equation}
An immediate consequence of this definition is that
the induced map $\rho:\Ga(A)\too\Ga(TM)$ is a Lie algebra
homomorphism
and for any $x\in
M$, there is an induced Lie bracket  on
${\G}_x={\mathrm Ker}(\rho_x)\subset A_x$ which makes it
into a Lie algebra.  

In this paper, we deal mostly with transitive Lie algebroids, i.e., Lie algebroids $(A,M,\rho)$ such that $\rho$ is surjective.	 In this case if $\G=\ker\rho$ then $\pi_\G:\G\too M$ is a  Lie algebroid with vanishing anchor called the adjoint Lie algebroid of $A$ and we have an exact sequence of Lie algebroids called {\it Atiyah sequence}
\begin{equation}\label{eq2}0\too\G\too A\stackrel{\rho}\too TM\too0.\end{equation}We denote by $\br_\G$ the  induced Lie bracket on $\Ga(\G)$.

A {\it splitting} of $A$ is a splitting of the Atiyah sequence, i.e., a vector bundle homomorphism $\ga:TM\too A$ such that $\rho\circ\ga=\mathrm{Id}_{TM}$. This determines a connection $\na^\ga$ on  $\G$ and $\Om^\ga\in\Om^2(M,\G)$ (the curvature of $\ga$) by
\begin{equation} \label{eq3}
\na_{X}^\ga U=[\ga(X),U]_A\esp \Om^\ga(X,Y)=\ga([X,Y])-[\ga(X),\ga(Y)]_A,\quad X,Y\in\Ga(TM),U\in\Ga(\G).
\end{equation} 
The following relations are immediate consequences of the Jacobi identity applied to $[\;,\;]_A$:
\begin{eqnarray}
&&\na^\ga_X[U,V]_\G=[\na^\ga_XU,V]_\G+[U,\na^\ga_XV]_\G,\label{eq4}\\ 
&&R^{\na^\ga}(X,Y)U:=\na^\ga_{[X,Y]}U-\na^\ga_X\na^\ga_YU+\na^\ga_Y\na^\ga_XU=[\Om^\ga(X,Y),U]_\G,\;\label{eq5}\\
&&d^{\na^\ga}\Om^\ga(X,Y,Z):=\oint\left(\na_X^\ga\Om^\ga(Y,Z)-\Om^\ga([X,Y],Z) \right)=0,\label{eq6}
\end{eqnarray}where $\oint$ stands for the cyclic sum, $U,V\in\Ga(\G)$ and $X,Y,Z\in\Ga(TM)$. As a consequence of \eqref{eq4}, one can deduce that if $x,y\in M$ and $\mu$ a path joining $x$ to $y$ then the parallel transport along $\mu$ with respect to $\na^\ga$, $\tau_\mu:\G_x\too\G_y$ is an isomorphism of Lie algebras. So if $M$ is connected then the fibers of $\G$ are isomorphic as Lie algebras. 

Conversely, given a  Lie algebroid $\pi_\G:\G\too M$ with vanishing anchor, a connection $\na$ on $\G$ and $\Om\in \Om^2(M,\G)$ satisfying \eqref{eq4}-\eqref{eq6} then $A=TM\oplus\G$ with the anchor $\mathrm{Id}_{TM}\oplus 0$ and the Lie bracket on $\Ga(A)$ given by
\begin{equation}\label{eq8} [X+U,Y+V]_A=[X,Y]+\{\Om(Y,X)+[U,V]_\G+\na_XV-\na_YU\},\quad X,Y\in\Ga(TM),U,V\in\Ga(\G) \end{equation}
is a transitive Lie algebroid.

A  transitive Euclidean Lie algebroid is a transitive Lie algebroid $(A,M,\rho)$ together with
an Euclidean product $\prsa$  on the vector bundle $\pi_A:A\too M$. There are two important objects naturally associated to  $(A,M,\rho,\prsa)$.
\begin{enumerate}
	\item A splitting of the Atiyah sequence of $A$ which induces a Riemannian metric on $M$. Indeed, For any $x\in M$, we denote by $\G_x^\perp$ the orthogonal of $\G_x$ with respect to $\prsa$ thus $$A=\G\oplus\G^\perp.$$ The restriction of 
	$\rho$ to $\G^{\perp}$ is an isomorphism onto $TM$ and its inverse $\ga:TM\too \G^\perp$ defines a splitting of the Atiyah sequence. We denote by $\Om^\ga$ and $\na^\ga$ the associated curvature and connection defined by \eqref{eq3}. It follows from what above that $A$ as a transitive Euclidean Lie algebroid is canonically isomorphic to $TM\oplus\G$ with the Lie bracket given by \eqref{eq8}, the anchor $\mathrm{Id}_{TM}\oplus0$ and the Euclidean product $\prsm\oplus\prs_\G$, where $\prs_\G$ is the restriction of $\prs_A$ to $\G$ and $\prsm$ is the Riemannian metric on $M$ given by $\langle X,Y\rangle_{TM}=\langle \ga(X),\ga(Y)\rangle_A$.
	
	\item The analogous of the Levi-Civita
	connection. Indeed,  the Koszul formula
	\begin{eqnarray}
	2\langle\D_ab,c\rangle_A&=&{\rho}(a).\langle b,c\rangle_A+{\rho}(b).\langle a,c\rangle_A-
	{\rho}(c).\langle a,b\rangle_A\label{koszul}\\
	&&	+\langle[c,a]_A,b\rangle_A+\langle[c,b]_A,a\rangle_A+\langle[a,b]_A,c\rangle_A,\quad a,b,c\in\Ga(A)\nonumber\end{eqnarray}defines a $\R$-bilinear map $\D:\Ga(A)\times\Ga(A)\too\Ga(A)$ characterized by the following three properties:
	\begin{enumerate}
		\item[$(i)$] $\D_{fa}b=f\D_ab$, $\D_a(fb)=\rho(a)(f)b+f\D_ab$, for any $a,b\in \Ga(A)$, $f\in C^\infty(M)$,
		\item[$(ii)$] $\D$ is metric, i.e., $\rho(a).\langle b,c\rangle_A=\langle\D_ab,c\rangle_A+\langle b,\D_ac\rangle_A$,
		\item[$(iii)$] $\D$ is torsion free, i.e., $\D_ab-\D_ba=[a,b]_A.$
	\end{enumerate}
	Thus $\D$ is a connection on the Lie algebroid $A$
	 well-known as the \emph{ Levi-Civita connection} associated to
	the Euclidean Lie algebroid $A$. The reader can consult \cite{boucetta, fernandes} for a detailed study of connections on Lie algebroids.

\end{enumerate}

For our purpose, we extract from the splitting $\ga$ and $\D$  the  necessary ingredients for defining generalized Cheeger-Gromoll metrics on the Euclidean bundle  $A$. Indeed,  we have already defined a Riemannian metric $\prsm$ on $M$. We define now a connection $\na^A$ on the vector bundle $A$  by
\begin{equation}\label{na}  \na^A_Xa=\D_{\ga(X)}a,\quad  X\in\Ga(TM), a\in\Ga(A),  \end{equation}and since $\D$ is metric, $\na^A$ preserves $\prsa$. Hence, according to the last section, we can define the family of generalized Cheeger-Gromoll metrics on $A$. The O'Neill shape tensor of these metrics is given by the curvature of $\na^A$. From the definition of $\na^A$, it is clear that $R^{\na^A}$ is an invariant of the transitive Euclidean Lie algebroid structure. 

\begin{Def} Let $A$ be a  transitive Euclidean Lie algebroid. We call the tensor field $R^{\na^A}$  principal curvature of $A$.
	
\end{Def}

In order to compute the principal curvature we need to explicit the expression of the Levi-Civita connection $\D$.

Thank to the splitting $\ga$, $\D$ can be computed by the means of the Levi-Civita connection $\na^M$ of $(M,\prsm)$, the Levi-Civita product
$\widehat{\D}$ 
associated  to $(\G,\prs_\G)$ and given by
\begin{equation}\label{eq7}
2\langle \widehat{\D}_UV,W\rangle_\G=\langle[U,V]_\G,W\rangle_\G+
\langle[W,V]_\G,U\rangle_\G+\langle[W,U]_\G,V\rangle_\G,\quad U,V,W\in\Ga(\G),
\end{equation} and the analogous O'Neill tensors \cite{one} (see \cite{bes}
for a detailed presentation)
$T$ and $H$  elements of $\Ga(A^*\otimes A^*\otimes A)$ whose values on two sections
$a,b\in\Ga(A)$ are given by
\begin{equation}\label{o}T_ab=(\D_{a^t}b^t)^\perp+(\D_{a^t}b^\perp)^t\quad\mbox{and}\quad H_ab=(\D_{a^\perp}b^t)^\perp+(\D_{a^\perp}b^\perp)^t,\end{equation}where $a^t$ is the projection on $\G$ and $a^\perp$ is the projection on $\G^\perp$. 
Indeed, we have the following relations which sum up all the properties of $\D$, $H$ and $T$ and which are easy to establish using Koszul formula \eqref{koszul}.
\begin{pr} \label{pr1} For any $a,b\in\Ga(A)$, any $U,V\in\Ga(\G)$ and any $X\in\Ga(TM)$, we have
	\begin{enumerate}
		\item[$(i)$] $T_{a^\perp}=H_{a^t}=0$, $H_{a^\perp}b^\perp\in\Ga(\G)$, $H_{a^\perp}b^t\in\Ga(\G^\perp)$, $T_{a^t}b^t\in\Ga(\G^\perp)$, $T_{a^t}b^\perp\in\Ga(\G)$,
		\item[$(ii)$] $H_{a^\perp}b^\perp=\frac12[a^\perp,b^\perp]_A^t=-\frac12\Om^\ga(\rho(a^\perp),\rho(b^\perp))$, 
		$\langle T_UV,\ga(X)\rangle_A=-\frac12\na^\ga_X(\prs_\G)(U,V)$,
		\item[$(iii)$] $\langle H_{a^\perp}b^t,c^\perp\rangle_A=-\langle H_{a^\perp}c^\perp,b^t\rangle_A,\;
		\langle T_{a^t}b^\perp,c^t\rangle_A=-\langle T_{a^t}c^t,b^\perp\rangle_A$,
		\item[$(iv)$] $\D_UV=\widehat{\D}_UV+T_UV$, $\D_{a^\perp}b^\perp=\ga(\na^M_{\rho(a^\perp)}\rho(b^\perp))+H_{a^\perp}b^\perp$,
		\item[$(v)$] $\D_{a^t}b^\perp=H_{b^\perp}a^t+T_{a^t}b^\perp$, $\D_{a^\perp}b^t=[a^\perp,b^t]_A+\D_{b^t}a^\perp$.
	\end{enumerate}
	
\end{pr}

From Proposition \ref{pr1}, one can deduce easily the following relations between the connexions $\na^A$, $\na^M$ and $\na^\ga$:
\begin{eqnarray}
	\na_{X}^AY^\ga &=&(\na_X^MY)^\ga+H_{X^\ga}Y^\ga=(\na_X^MY)^\ga-\frac12\Om^\ga(X,Y),\nonumber\\
	\na^A_X\al&=&\na^\ga_X\al+T_{\al}X^\ga+H_{X^\ga}\al,\quad X,Y\in\Ga(TM),\al\in\Ga(\G),\label{ec}
\end{eqnarray}	where  $X^\ga=\ga(X)$. 
  		
  Having these formulas in mind we can compute now the expression of the principal curvature of $A$.

  \begin{pr}\label{fb} We have, for any $X,Y,Z\in\Ga(TM)$ and $U\in\Ga(\G)$,
  	\begin{eqnarray*}
  		R^{\na^A}(X,Y)Z^\ga&=&\left\{(R^M(X,Y)Z)^\ga+H_{Y^\ga}H_{X^\ga}Z^\ga-H_{X^\ga}H_{Y^\ga}Z^\ga\right\}+\left\{T_{H_{X^\ga}Z^\ga}Y^\ga-T_{H_{Y^\ga}Z^\ga}X^\ga -\frac12 \na_Z^{M,\ga}\Om^\ga(X,Y)\right\},\\
  		(R^{\na^A}(X,Y)U)^t&=&
  		R^{\na^\ga}(X,Y)U+H_{Y^\ga}H_{X^\ga}U-H_{X^\ga}H_{Y^\ga}U+T_{U}[X,Y]^\ga\\&&-T_{\na^\ga_{Y}U}X^\ga-\na^\ga_{X}T_{U}Y^\ga-T_{T_{U}Y^\ga}X^\ga
  		+T_{\na^\ga_{X}U}Y^\ga+\na^\ga_{Y}T_{U}X^\ga+T_{T_{U}X^\ga}Y^\ga,\\
  		\langle R^{\na^A}(X,Y)U,Z^\ga\rangle_A&=&-\langle R^{\na^A}(X,Y)Z^\ga,U\rangle_A,
  	\end{eqnarray*}where $R^M$ is the curvature of $\na^M$ and $R^{\na^\ga}$ is the curvature of $\na^\ga$ and
  	$$\na_Z^{M,\ga}\Om^\ga(X,Y)=	\na^\ga_{Z}\Om^\ga(X,Y)-\Om^\ga(X,\na^M_ZY)-\Om^\ga(\na^M_ZX,Y).$$
  	
  \end{pr}

  \begin{proof} 
  	It is a straightforward computation using  \eqref{ec}. Indeed,
  	\begin{eqnarray*}
  		R^{\na^A}(X,Y)Z^\ga
  		&=&\na^A_{[X,Y]}Z^\ga-\na^A_{X}\na^A_{Y}Z^\ga+\na^A_{Y}\na^A_{X}Z^\ga\\
  		&=&(\na^M_{[X,Y]}Z)^\ga-\frac12\Om^\ga([X,Y],Z)-\na^A_{X}(\na^M_YZ)^\ga+\frac12\na^A_{X}
  		\Om^\ga(Y,Z)+\na^A_{Y}(\na^M_XZ)^\ga-\frac12\na^A_{Y}
  		\Om^\ga(X,Z)\\
  		&=&(R^M(X,Y)Z)^\ga-\frac12\Om^\ga([X,Y],Z)+\frac12\Om^\ga(X,\na^M_YZ)-\frac12\Om^\ga(Y,\na^M_XZ)\\
  		&&+\frac12\na^{\ga}_{X}\Om^\ga(Y,Z)-T_{H_{Y^\ga}Z^\ga}X^\ga-H_{X^\ga}H_{Y^\ga}Z^\ga
  		-\frac12\na^{\ga}_{Y}\Om^\ga(X,Z)+T_{H_{X^\ga}Z^\ga}Y^\ga+H_{Y^\ga}H_{X^\ga}Z^\ga.
  	\end{eqnarray*}Now put
  	\[ Q=-\frac12\Om^\ga([X,Y],Z)+\frac12\Om^\ga(X,\na^M_YZ)-\frac12\Om^\ga(Y,\na^M_XZ)+\frac12\na^{\ga}_{X}\Om^\ga(Y,Z)-\frac12\na^{\ga}_{Y}\Om^\ga(X,Z). \]We have
  	\begin{eqnarray*}
  		2Q&=&-\Om^\ga([X,Y],Z)-\Om^\ga([Y,Z],X)+\Om^\ga(X,\na^M_ZY)-\Om^\ga([Z,X],Y)\\&&+
  		\Om^\ga(\na^M_ZX,Y)+\na^{\ga}_{X}\Om^\ga(Y,Z)+\na^{\ga}_{Y}\Om^\ga(Z,X)\\
  		&\stackrel{\eqref{eq6}}=&-\na^\ga_{Z}\Om^\ga(X,Y)+\Om^\ga(X,\na^M_ZY)+\Om^\ga(\na^M_ZX,Y)\\
  		&=&-\na_Z^{M,\ga}\Om^\ga(X,Y).
  	\end{eqnarray*}Finally,
  	\[ R^{\na^A}(X,Y)Z^\ga=(R^M(X,Y)Z)^\ga+H_{Y^\ga}H_{X^\ga}Z^\ga-H_{X^\ga}H_{Y^\ga}Z^\ga+
  	\left\{T_{H_{X^\ga}Z^\ga}Y^\ga-T_{H_{Y^\ga}Z^\ga}X^\ga -\frac12 \na_Z^{M,\ga}\Om^\ga(X,Y)  \right\}. \]
  	Let compute now $R^{\na^A}(X,Y)U$ for $U\in\Ga(\G)$. Put $R^{\na^A}(X,Y)U=(R^{\na^A}(X,Y)U)^\perp+(R^{\na^A}(X,Y)U)^t$. Since, for any $Z\in\Ga(TM)$, 
  	\[ \langle (R^{\na^A}(X,Y)U)^\perp,Z^\ga\rangle_A=\langle R^{\na^A}(X,Y)U,Z^\ga\rangle_A=-\langle R^{\na^A}(X,Y)Z^\ga,U\rangle_A, \]it suffices to compute $(R^{\na^A}(X,Y)U)^t$. Now
  	\begin{eqnarray*}
  		R^{\na^A}(X,Y)U
  		&=&\na^A_{[X,Y]}U-\na^A_{X}\na^A_{Y}U+\na^A_{Y}\na^A_{X}U\\
  		&=&\na^\ga_{[X,Y]}U+T_{U}[X,Y]^\ga+H_{[X,Y]^\ga}U-\na^A_{X}\na^\ga_{Y}U-\na^A_{X}T_{U}Y^\ga-\na^A_{X}H_{Y^\ga}U\\&&
  		+\na^A_{Y}\na^\ga_{X}U+\na^A_{Y}T_{U}X^\ga+\na^A_{Y}H_{X^\ga}U\\
  		&=&R^{\na^\ga}(X,Y)U+T_{U}[X,Y]^\ga+H_{[X,Y]^\ga}U-T_{\na^\ga_{Y}U}X^\ga-H_{X^\ga}\na^\ga_{Y}U\\&&-\na^\ga_{X}T_{U}Y^\ga-T_{T_{U}Y^\ga}X^\ga-H_{X^\ga}T_{U}Y^\ga
  		-\na^A_{X}H_{Y^\ga}U
  		+T_{\na^\ga_{X}U}Y^\ga+H_{Y^\ga}\na^\ga_{X}U\\&&+\na^\ga_{Y}T_{U}X^\ga+T_{T_{U}X^\ga}Y^\ga+
  		H_{Y^\ga}T_{U}X^\ga+\na^A_{Y}H_{X^\ga}U.
  		\end{eqnarray*}	
  		From \eqref{ec}, we have $(\na^A_{X}H_{Y^\ga}U)^t=H_{X^\ga}H_{Y^\ga}U$ and 
  		$(\na^A_{Y}H_{X^\ga}U)^t=H_{Y^\ga}H_{X^\ga}U$ and hence
  \begin{eqnarray*} (R^{\na^A}(X,Y)U)^t&=&
  R^{\na^\ga}(X,Y)U+H_{Y^\ga}H_{X^\ga}U-H_{X^\ga}H_{Y^\ga}U+T_{U}[X,Y]^\ga\\&&-T_{\na^\ga_{Y}U}X^\ga-\na^\ga_{X}T_{U}Y^\ga-T_{T_{U}Y^\ga}X^\ga
  +T_{\na^\ga_{X}U}Y^\ga+\na^\ga_{Y}T_{U}X^\ga+T_{T_{U}X^\ga}Y^\ga. \end{eqnarray*}		This completes the computation.	
  \end{proof}
  
  \begin{remark} 
  	The expression of $R^{\na^A}$ given above shows that it involves the curvature of $M$, the curvature of the connection $\na^\ga$ on $\G$ and the curvature of the splitting $\ga$. This creates a rich geometric situation. The following proposition enhances this fact.

  \end{remark}

  \begin{pr}\label{v} Let $(A,M,\rho,\prs_A)$ be a transitive Euclidean Lie algebroid and denote by $\ric^M$ and $s^M$, respectively, the Ricci curvature and the scalar curvature of $(M,\prsm)$. Then:
  	\label{fb1}\begin{enumerate}
  		\item[$(i)$]  If $R^{\na^A}=0$ then, for any $X,Y\in\Ga(TM)$,
  		\[ \langle R^M(X,Y)X,Y\rangle_{TM}=\langle H_{X^\ga}Y^\ga,H_{X^\ga}Y^\ga\rangle_A\geq0,\;\; 
  		\ric^M(X,Y)=\langle H_{X^\ga},H_{Y^\ga}\rangle_A\esp s^M=|H|^2 , \]where
  		\[ \langle H_{X^\ga},H_{Y^\ga}\rangle_A=\sum_{i=1}^n\langle H_{X^\ga}{E_i^\ga},H_{Y^\ga}{E_i^\ga}\rangle_A\esp |H|^2=\sum_{i=1}^n\langle H_{E_i^\ga},H_{E_i^\ga}\rangle_A \]and $(E_1,\ldots,E_n)$ is a local orthonormal frame on $M$.
  		In this case the sectional curvature of $(M,\prsm)$ is non-negative and it vanishes if and only if  $R^M=0$.
  		\item[$(ii)$] If $R^{\na^A}=0$ and $T=0$ then $\na^{M,\ga}\Om^\ga=0$, $\na^{M,\ga}(R^{\na^\ga})=0$  and $\na^M(R^M)=0$. In particular,
  		$(M,\prsm)$ is locally symmetric. 
  	\end{enumerate}
  	
  \end{pr}
  \begin{proof} From Proposition \ref{fb}, if $R^{\na^A}$ vanishes then the vertical and the horizontal part of $R^{\na^A}(X,Y)Z^\ga$ vanish, i.e.,
  	\[ (R^M(X,Y)Z)^\ga+H_{Y^\ga}H_{X^\ga}Z^\ga-H_{X^\ga}H_{Y^\ga}Z^\ga=0 \esp
  	\na_Z^{M,\ga}\Om^\ga(X,Y)=2T_{H_{X^\ga}Z^\ga}Y^\ga-2T_{H_{Y^\ga}Z^\ga}X^\ga. \]
  	Since $H_{X^\ga}$ is skew-symmetric with respect to $\prs_A$, we get 
  	\[ \langle R^M(X,Y)Z,S\rangle_{TM}=\langle H_{X^\ga}Z^\ga,H_{Y^\ga}S^\ga\rangle_\G-
  	\langle H_{Y^\ga}Z^\ga,H_{X^\ga}S^\ga\rangle_\G,\eqno(*) \]and the formulas in part $(i)$ follow. Moreover, if $T=0$ then from the second formula, we get that $\na_Z^{M,\ga}\Om^\ga=0$, i.e., 
  	$\Om^\ga$ is parallel with respect the connections $\na^M$ and $\na^\ga$. This means that if $c$ is a curve in $M$ joining two points $x,y$,  $\tau^M:T_xM\too T_yM$ the parallel transport along $c$ associated to $\na^M$ and $\tau^\ga:\G_x\too \G_y$ the parallel transport along $c$ associated to $\na^\ga$ then
  	\[ \Om^\ga(\tau^MX,\tau^MY)=\tau^\ga\Om^\ga(X,Y). \]
  	Note that since $T=0$, we have from Proposition \ref{pr1} that $\na^\ga(\prs_\G)=0$ and hence $\tau^\ga$ is an isometry.
  	So from $(*)$ we get, since $H_{X^\ga}Y^\ga=-2\Om^\ga(X,Y),$
  	\begin{eqnarray*}
  		\langle R^M(\tau^MX,\tau^MY)\tau^MZ,\tau^MS\rangle_{TM}&=&\langle \tau^\ga H_{X^\ga}Z^\ga, \tau^\ga H_{Y^\ga}S^\ga\rangle_\G-
  		\langle\tau^\ga H_{Y^\ga}Z^\ga,\tau^\ga H_{X^\ga}S^\ga\rangle_\G\\
  		&=&\langle H_{X^\ga}Z^\ga,H_{Y^\ga}S^\ga\rangle_\G-
  		\langle H_{Y^\ga}Z^\ga,H_{X^\ga}S^\ga\rangle_\G\\
  		&=&\langle R^M(X,Y)Z,S\rangle_{TM}.
  	\end{eqnarray*} Thus $(\tau^M)^{-1}R^M(\tau^MX,\tau^MY)\tau^MZ=R^M(X,Y)Z$ which shows that $R^M$ is parallel with respect to $\na^M$. A same argument shows that $\na^{M,\ga}(R^{\na^\ga})=0$ and completes the proof.
  \end{proof}

  The following theorem sum up all what we have seen so far in this section and gives all what one needs to know in order to study natural metrics, in general, and generalized Cheeger-Gromoll metrics, in particular, on transitive Euclidean Lie algebroids.
  
  \begin{theo}\label{main6} Any transitive Euclidean Lie algebroid $A$ can be canonically identified to $TM\oplus\G$ where $(M,\prsm)$ is  Riemannian manifold, $\G$ a  vector bundle of Lie algebras endowed with an Euclidean product $\prs_\G$. The Euclidean product on $A$ is given by $\prsm\oplus \prs_\G$, the anchor by $\mathrm{Id}_{TM}\oplus 0$ and the Lie bracket is given by
  \eqref{eq8}, where $\na$ is a linear connection on $\G$ and $\Om\in\Om^2(M,\G)$ satisfying \eqref{eq4}-\eqref{eq6}. Moreover, let $T$ and $H$ in $\Ga(A^*\otimes A\otimes A)$ given,  for any $X,Y\in\Ga(TM)$ and $U,V\in\Ga(\G)$, by
  \begin{eqnarray*}
  H_U&=&0,\; H_XY=-\frac12\Om(X,Y),\; \langle H_XU,Y\rangle_{TM}=-\langle U,H_XY\rangle_\G,\\
  T_X&=&0,\; \langle T_UV,X\rangle_{TM}=-\frac12\na_{X}(\prs_\G)(U,V),\; \langle T_UX,V\rangle_\G=-\langle X,T_UV\rangle_{TM}.
  \end{eqnarray*}The connexion $\na^A$ on $A$ given by
  \[ \na_X^AY=\na_X^MY+H_XY\esp \na_X^AU=\na_XU+T_UX+H_XU \]preserves the Euclidean product on $A$ and its curvature is given, 
  for any $X,Y,Z\in\Ga(TM)$ and $U\in\Ga(\G)$, by
  \begin{eqnarray*}
  	R^{\na^A}(X,Y)Z&=&\left\{R^M(X,Y)Z+H_{Y}H_{X}Z-H_{X}H_{Y}Z\right\}+\left\{T_{H_{X}Z}Y-T_{H_{Y}Z}X -\frac12 \na_Z^{M,\na}\Om(X,Y) \right\},\\
  	(R^{\na^A}(X,Y)U)^t&=&
  	R^{\na}(X,Y)U+H_{Y}H_{X}U-H_{X}H_{Y}U+T_{U}[X,Y]\\&&-T_{\na_{Y}U}X-\na_{X}T_{U}Y-T_{T_{U}Y}
  	X
  	+T_{\na_{X}U}Y+\na_{Y}T_{U}X+T_{T_{U}X}Y,\\
  	\langle R^{\na^A}(X,Y)U,Z\rangle_A&=&-\langle R^{\na^A}(X,Y)Z,U\rangle_A,
  \end{eqnarray*}where $R^M$ is the curvature of $\na^M$ and $R^{\na}$ is the curvature of $\na$ and
  $$\na_Z^{M,\na}\Om(X,Y)=	\na_{Z}\Om(X,Y)-\Om(X,\na^M_ZY)-\Om(\na^M_ZX,Y).$$
  
  \end{theo}

  \begin{remark}\label{rem2} It is important to notice that the definition of $\na^A$ in Theorem \ref{main6} doesn't involve the Lie bracket on $\Ga(\G)$ and, in the proof of Proposition \ref{fb}, to get the expression of $R^{\na^A}$ we have used only the fact that $d^\na\Om=0$. So if one is not interested about the Lie algebroid bracket, he can build an Euclidean vector bundle with a connection and a parallel  Euclidean metric by considering a Riemannian manifold $(M,\prsm)$, an Euclidean vector bundle $B\too M$ with a connection $\na$ and $\Om\in\Om^2(M,B)$ satisfying $d^\na\Om=0$ and build $A=TM\oplus B$ with the Euclidean product $\prsm\oplus\prs_B$ and the connection $\na^A$ as in Theorem \ref{main6}. The curvature of $\na^A$ is given as in Theorem \ref{main6}. This gives a more general situation where one can study natural metrics.
  	
  \end{remark}

  \section{ Characterization of  Atiyah Euclidean Lie algebroids}\label{section4}

  Atiyah Lie algebroids associated to principal bundles constitute a large class of transitive Lie algebroids and, actually, any integrable transitive Lie algebroid in the sense of being the Lie algebroid of a Lie groupoid is an Atiyah Lie algebroid (see \cite{Li}). When endowed with an Euclidean product these Euclidean Lie algebroids could be build as in Theorem \ref{main6}. We devote this section to give a precise description of Atiyah Euclidean Lie algebroids in the spirit of Theorem \ref{main6}. In particular, we will give a precise description of the Atiyah Lie algebroid associated to the principal bundle of orthonormal frames over a Riemannian manifold and we will show that it carries  a natural family of Euclidean products which make it an ideal candidate for carrying natural metrics. 
  
  One can consult \cite{kubarski} for a detailed treatment of Atiyah Lie algebroids.

  Through this section $P(M,\pir,G)$ is a principal $G$-bundle $\pir : P\too M$. Let start by defining the Atiyah Lie algebroid associated to $P(M,\pir,G)$.

  Let $\Ga(TP)^G$ 	and $\Ga(\mathcal{V}P)^G$  denote, respectively, the $C^\infty(M)$-module of $G$-invariant vector fields on $P$ and its subspace of vertical vector fields. Any vector field in $\Ga(TP)^G$ is $\pir$-projectable on a vector field on $M$ and we have an exact sequence of $C^\infty(M)$-modules
  \begin{equation}\label{sec}
  0\too \Ga(\mathcal{V}P)^G\too \Ga(TP)^G\stackrel{d\pir}\too\Ga(TM)\too 0,
  \end{equation} which is also an exact sequence of real Lie algebras.
  
  Let $\g = Lie(G)$ endowed with the Lie bracket $\br_\g$ obtained from the identification of $\g$ with the space of left invariant vector fields and denote by $R_a :  u\in P\mapsto u.a \in P$ the diffeomorphism that is induced by the right action of $a \in G$ on $P$. Consider $P\times_G\g$ the quotient of $P\times\g$ by the action $a.(u,\kappa)=(u.a,\Ad_{a^{-1}}\kappa)$. This is  a vector bundle over $M$ and we denote by $\pir_0:P\times_G\g\too M$ the natural projection. We identify $\Ga(P\times_G\g)$ with the space $C^\infty(P,\g)^G$ of smooth applications $s:P\too\g$ satisfying
  $s(u.a)=\Ad_{a^{-1}}s(u)$ for any $a\in G$ and any $u\in P$. 
  We define $$V:C^\infty(P,\g)^G\too \Ga(\mathcal{V}P)^G,\, s\mapsto V^s,$$ where $V^s$ is the complete vector field on $P$ whose flow $\phi^s$ is given by $\phi^s(t,u)=u.\exp(-ts(u))$. The map $V$ defines an isomorphism of vector space which is, actually, an isomorphism of Lie algebras. Indeed, we have the following formulas which are part of the folklore:
  \begin{equation}\label{eq23}
  [V^{s_1},V^{s_2}]=V^{[s_1,s_2]_\g},\; [U,V^{s}] =V^{U(s)}\esp V^{s_1}(s_2)=[s_1,s_2]_\g,\; 
  \end{equation}where $s,s_1,s_2\in C^\infty(P,\g)^G, U\in \Ga(TP)^G$.We have chosen a minus sign in the definition of $\phi^s$ in order to avoid a minus sign in the first formula above.

  There exists an unique Lie algebroid, up to an isomorphism, whose the  exact sequence of Lie algebras associated to its Atiyah sequence is isomorphic to \eqref{sec}.
  Indeed,
  over any point $m \in M$, we define an equivalence relation of tangent vectors to $P$. If $\pir(u) = m$, $a \in G$, $X_u \in T_uP$, and $X_{u.a} \in T_{u.a}P$, the vectors $X_u$ and $X_{u.a}$ are said to be equivalent if and only if $X_{u.a} = (T_uR_a)(X_u)$. The equivalence classes of this relation form a vector space $A_m$ isomorphic to $T_uP$, and the disjoint union $A = \bigcup_{m\in M }A_m$ is a vector bundle $\pi_A:A\too M$ of rank $\dim P=\dim M+\dim G$.
  
  Since $\pir \circ R_a = \pir$ , it is clear that the image of a tangent vector $X_u$ by the surjection $T_u\pir : T_uP \too T_mM$ does not depend on the representative $X_u$ of the class $[X_u] \in A_m$. Hence, we get a well-defined surjection $\rho_m : A_m \too T_mM$, as well as a surjective bundle map $\rho : A \too T M$ over the identity. This map will be the anchor of the Atiyah algebroid  associated with the principal bundle $P(M,\pir,G)$. 
  
  The map $\tau: \Ga(TP)^G\too \Ga(A)$, $U\mapsto \tau(U)$, where $\tau(U)(\pir(u))=[U(u)]$ is an isomorphism of $C^\infty(M)$-modules and hence there exists an unique Lie bracket $\br_A$ on $\Ga(A)$ such that $\tau$ is an isomorphism of Lie algebras and $\tau(\Ga(\mathcal{V}P)^G)=\Ga(\G)$ where $\G$ is the adjoint Lie algebroid of $A$.
  We get a transitive Lie  algebroid $(A,M,\rho,\br_A)$ known as the Atiyah algebroid of $P(M,\pir,G)$. 
  
  Let $\ga:TM\too A$ be a splitting of the Atiyah Lie algebroid and consider $\na^\ga$ and $\Om^\ga\in\Om^2(M,\G)$ which are defined by \eqref{eq3}. It defines a $C^\infty(M)$-module homomorphism $\ga:\Ga(TM)\too \Ga(A)$ and hence a splitting $\tau^{-1}\circ\ga:\Ga(TM)\too\Ga(TP)^G$ of \eqref{sec}. Put, for any $U=\tau^{-1}\circ\ga(X)$ and $s\in C^\infty(P,\g)^G$
  \[ \om(U)=0\esp \om(V^s)=-s. \]
  One can check easily that this defines $\om\in\Om^1(P,\g)$ which is $G$-invariant and hence a connection 1-form on $P$. The associated $G$-invariant horizontal distribution is given by $\mathcal{H}TP(u)=\{ \tau^{-1}\circ\ga(X)(u): X\in\Ga(TM)  \}$. 
  For any $X\in\Ga(TM)$, denote by $X^{\om}=\tau^{-1}\circ\ga(X)$ which is the horizontal left of $X$. For any, $X,Y\in\Ga(TM)$, we have
  \begin{eqnarray*}
  d\om(X^{\om},Y^{\om})&=&-\om([X^{\om},Y^{\om}])\\&=&-\om(\tau^{-1}([\ga(X),\ga(Y)]_A))\\
  &\stackrel{\eqref{eq3}}=&\om(\tau^{-1}(\Om^\ga(X,Y)))\\
  &=&-K(X,Y),
  \end{eqnarray*}where $K(X,Y)$ is the unique element of $C^\infty(P,\g)^G$ satisfying $\tau(V^{K(X,Y)})=\Om^\ga(X,Y)$. Moreover, for any $s\in C^\infty(P,\g)^G$ and any $X\in\Ga(TM)$,
  \[ \na^\ga_X\tau(V^s)\stackrel{\eqref{eq3}}=[\ga(X),\tau(V^s)]_A=[\tau(X^{\om}),\tau(V^s)]=
  \tau([X^{\om},V^s])\stackrel{\eqref{eq23}}=\tau(V^{X^{\om}(s)}). \] The splitting $\ga$ defines an identification of $A$ with $TM\oplus\G$ with the Lie bracket given by \eqref{eq8}. But $\G$ can be identified to $P\times_G\g$ by the mean of $\tau^{-1}:\Ga(\G)\too \Ga(\mathrm{V}TP)^G$ and $V^{-1}:\Ga(\mathrm{V}TP)^G\too C^{\infty}(P,\g)^G$. The following result sum up all what we have seen so far.
  \begin{pr}\label{atiyah1}\begin{enumerate}
  		\item There is a correspondence between the splittings of the Atiyah Lie algebroid $A$ and the 1-form connections of $P$.
  	\item 	For any splitting $\ga:TM\too A$ of the Atiyah Lie algebroid of $P(M,\pir,G)$ there exists a connection 1-form $\om\in\Om^1(P,\g)$ such that $A$ is isomorphic to $TM\oplus P\times_G\g$ with the anchor $\mathrm{Id}_{TM}\oplus 0$ and the Lie bracket given by
  		\begin{equation}\label{eq67} [X+s_1,Y+s_2]_A=[X,Y]+\left\{ d\om(X^\om,Y^\om)+X^\om(s_2)-Y^\om(s_1)+[s_1,s_2]_\g      \right\}, \end{equation}where 
  		$X,Y\in\Ga(TM), s_1,s_2\in C^\infty(P,\g)^G$.
 		\item For any 1-from connection $\om$ on $P$, the bracket given by \eqref{eq67} defines a Lie algebroid structure on $TM\oplus P\times_G\g$ which is isomorphic to the Atiyah Lie algebroid of $P(M,\pir,G)$. 
  	\end{enumerate}

  \end{pr}
  
  Since any Euclidean product on $A$ comes with a splitting and an Euclidean product on $\G$. When we identify $\G$ to $P\times_G\g$, we get also an Euclidean product on $P\times_G\g$ which is entirely determined by smooth map $h:P\too\otimes^2\g^*$, $u\mapsto h_u$ with $h_u$ is an Euclidean product on $\g$ satisfying
  $h_{u.a}=\Ad_{a^{-1}}h_u$. So, we get the following corollary.

  \begin{co}\label{atiyah} Let $\prs_A$ be an Euclidean product on the Atiyah Lie algebroid $A$ associated to $P(M,\pir,G)$. Then there exists a connection one-form $\om\in \Om^1(P,\g)$, a Riemannian metric $\prsm$ on $M$, an Euclidean product $h:P\too\otimes^2\g^*$ such that $(A,\prs_A)$ is canonically isomorphic as a  transitive Euclidean Lie algebroid to $TM\oplus (P\times_G\g)$ with the anchor $\mathrm{Id}_{TM}\oplus 0$,  the Lie bracket given by \eqref{eq67} and the Euclidean product given
  	\[ \langle X+s_1,Y+s_2\rangle_A(\pir(u))=\langle X,Y\rangle_{TM}(\pir(u))+h_u(s_1(u),s_2(u)),\quad X,Y\in\Ga(TM), s_1,s_2\in C^\infty(P,\g)^G. \]
  	Moreover, the tensor $H$ and $T$ defined in Theorem \ref{main6} are given by $H_XY=\frac12d\om(X^\om,Y^\om)$ and $\langle T_{s_1}s_2,X\rangle_{TM}=-\frac12\mathcal{L}_{X^\om}h(s_1,s_2),$ where
  	\[ \mathcal{L}_{X^\om}h(s_1,s_2)=X^\om.h(s_1,s_2)-h(X^\om(s_1),s_2)-h(s_1,X^\om(s_2)). \]
  	Conversely, any Riemannian metric on $M$, any connection 1-form on $P$ and any Euclidean product on $P\times_G\g$ define a transitive Euclidean Lie algebroid structure on $TM\oplus P\times_G\g$ as above.
  \end{co}
  
 As an application of Proposition \ref{atiyah1} and Corollary \ref{atiyah}, we describe now the Atiyah Lie algebroid of the $\mathrm{O}(n)$-principal bundle of orthonormal frames over a Riemannian manifold and we endow it with a family depending on one parameter of Euclidean product.
 
 \begin{theo} \label{on} Let $(M,\prsm)$ be a Riemannian manifold, $O(TM)$ the $\mathrm{O}(n)$-principal bundle of orthonormal frames over $M$ and $\mathrm{so}(TM)=\bigcup_{m\in M}\mathrm{so}(T_mM)$ where $\mathrm{so}(T_mM)$ is the Lie algebra of skew-symmetric endomorphisms of $T_mM$. Then the Atiyah Lie algebroid of $O(TM)$ is canonically isomorphic to $TM\oplus \mathrm{so}(TM)$ with anchor $\mathrm{Id}_{TM}\oplus 0$ and the Lie bracket given by \[ [X+F,Y+G]_A=[X,Y]+\left\{ \na_X^M(G)-\na_Y^M(F)+[F,G]-R^M(X,Y)\right\}, \]
 	where $F,G\in\Ga(\mathrm{so}(TM))$,  $X,Y\in \Ga(TM)$,  $R^M$ is the curvature of $\na^M$, a section of $\mathrm{so}(TM)$ is seen as a skew-symmetric bundle homomorphism and $[F,G]=F\circ G-G\circ F$.
 	
 \end{theo}
 \begin{proof}
Recall that $O(TM)$ over $M$ consisting of $(m,z)$ such that $z:\R^n\too T_mM$ is an isometry where $\R^n$ is induced with its canonical Euclidean metric,  $T_mM$ with $\prsm^m$ and $\pir:O(TM)\too M$, $(m,z)\mapsto m$. The Levi-Civita connection $\na^M$ of $M$ defines a connection 1-form $\om$ on $O(TM)$ which can be described as follows. For any $X\in\Ga(TM)$ and $s\in C^{\infty}(O(TM),\mathrm{so}(n))^{O(n)}$
\[ \om(X^\om)=0\esp \om(V^s)=-s, \]where 
 $X^\om$  is the vector field on $O(TM)$ given by
  \[ X^\om(m,z)=\frac{d}{dt}_{|t=0}\tau^{0,t}\circ z \]and $\tau^{0,t}:T_mM\too T_{\phi^X(m,t)}M$ is the parallel transport along the curve $s\too \phi^X(s,m)$ and $\phi^X$ is the flow of $X$. According to Proposition \ref{atiyah1}, $\om$ defines a Lie algebroid structure on $TM\oplus O(TM)\times_{\mathrm{O}(n)}\mathrm{so}(n)$ whose Lie bracket is given by \eqref{eq67}. Now $O(TM)\times_{\mathrm{O}(n)}\mathrm{so}(n)$ has a natural identification with $\mathrm{so}(TM)$  via $[(m,z),A]\mapsto z\circ A\circ z^{-1}$. Hence $TM\oplus \mathrm{so}(TM)$ carries a structure of Lie algebroid isomorphic to the Atiyah Lie algebroid of $O(TM)$ with anchor $\mathrm{Id}_{TM}\oplus 0$. Let compute the Lie bracket obtained from \eqref{eq67} when we identify $O(TM)\times_{\mathrm{O}(n)}\mathrm{so}(n)$ with $\mathrm{so}(TM)$. At the level of the space of sections the identification is given by
  \[ \Ga(\mathrm{so}(TM))\too C^{\infty}(O(TM),\mathrm{so}(n))^{O(n)},\quad F\mapsto s_F(m,z)=z^{-1}\circ F_m\circ z.  \]
  We have, for any $X\in\Ga(TM)$,
  \begin{eqnarray*}
  	X^\om(s_F)(m,z)&=&\frac{d}{dt}_{|t=0}s_F(\phi_t^X(m),\tau^{0,t}\circ z)\\
  	&=&\frac{d}{dt}_{|t=0}z^{-1}\circ (\tau^{0,t})^{-1}\circ F_{\phi_t^X(m)}\circ \tau^{0,t}\circ z\\
  	&=&z^{-1}\circ \na^M_{X(m)}(F)\circ z.
  \end{eqnarray*}On the other hand, we have, for any $X,Y\in\Ga(TM)$,
  \[ [X^\om,Y^\om]-[X,Y]^\om=-V^{\om([X^\om,Y^\om])}. \]
  By using \eqref{eq23} and the expression of $X^\om(s_F)$ obtained above, we get for any $F\in\Ga(\mathrm{so}(TM))$,
  \begin{eqnarray*}
  \;[s_F,\om([X^\om,Y^\om])]_{\mathrm{so}(n)}(m,z)&=&-V^{\om([X^\om,Y^\om])}(s_F)(m,z)\\
  &=&z^{-1}\circ\left( R^M_{X,Y}(F)  \right)\circ z,
  \end{eqnarray*}where $R^M_{X,Y}$ is the curvature of $\na^M$ as a connection on the vector bundle $\mathrm{so}(TM)$. Or $R^M_{X,Y}(F)=[R^M(X,Y),F]$. Thus
  \[ [s_F,\om([X^\om,Y^\om])]_{\mathrm{so}(n)}=[s_{R^M(X,Y)},s_F]_{\mathrm{so}(n)} \] and since the center of $\mathrm{so}(n)$ is trivial we get $\om([X^\om,Y^\om])=s_{R^M(X,Y)}$ which completes the proof.
  \end{proof}
  
  Since the vector bundle $\mathrm{so}(TM)$ has a natural Euclidean product obtained by considering the Killing form on each fiber, we get the following class of transitive Euclidean Lie algebroids associated naturally to any Riemannian manifold.
  
  \begin{Def}\label{def} Let $(M,\prsm)$ be a Riemannian manifold and $k>0$. We denote by $AO(M,k)$ the transitive Euclidean Lie algebroid $TM\oplus \mathrm{so}(TM)$ obtained in Theorem \ref{on} and endowed with the Euclidean product
  	\[ \langle X+F,Y+G\rangle_k=\langle X,Y\rangle_{TM}-k\tr(F\circ G). \]
  	We call $AO(M,k)$ the $k$-Atiyah Euclidean Lie algebroid of $(M,\prsm)$. Note that for $AO(M,k)$ we have $T=0$ and
  	\[ H_XY=-\frac12R^M(X,Y)\esp \langle H_XF,Y\rangle_{TM}=-\frac12k\tr(F\circ R^M(X,Y)). \]
  \end{Def}

  \section{Generalized Cheeger-Gromoll metrics on the  $k$-Atiyah Euclidean Lie algebroid over a space form}\label{section5}	
  
  This section is mainly devoted to give a proof of Theorem \ref{main5}. We consider a Riemannian manifold $(M,\prsm)$ of dimension $n$ and $AO(M,k)$ its $k$-Atiyah Euclidean Lie algebroid defined in Definition \ref{def}. The connection $\na^A$ on the Euclidean Lie algebroid $AO(M,k)$ defined in \eqref{na} is given by virtue of Theorem \ref{main6}, for any $X,Y\in\Ga(TM)$ and $F\in\Ga(\mathrm{so}(TM))$, by
  \[ \na^A_XY=\na^M_XY-\frac12R^M(X,Y)\esp \na^A_XF=\na_X^M(F)+H_XF, \langle H_XF,Y\rangle_{TM}=-\frac12k\tr(F\circ R^M(X,Y)). \]
  As in Section \ref{section2}, we endow $AO(M,k)$ with the family of generalized Cheeger-Gromoll metrics $h_{p,q}$ thanks to $\na^A$ (see \eqref{cg}). Recall that the O'Neill shape tensor of the Riemannian submersion $\pi:(AO(M,k),h_{p,q})\too (M,\prsm)$ is given by
  \[ B_{X^h}Y^h(a)=\frac12(R^{\na^A}(X,Y)a)^v, \]where $X^h$ is the vector field on $AO(M,k)$ the horizontal left of $X$.  Put
  \[ |B|^2=\sum_{i}h_{p,q}(B_{X_i^h},B_{X_i^h})\esp h_{p,q}(B_{X_i^h},B_{X_i^h})=\sum_{j\not=i}h_{p,q}(B_{X_i^h}X_j^h,B_{X_i^h}X_j^h), \]where $(X_i)_{i=1}^n$ is any local orthonormal frame on $M$.
  	
  For any $X,Y\in\Ga(TM)$, $X\wedge Y$ is the skew-symmetric endomorphism of $TM$ given by
  \[ X\wedge Y(Z)=\langle Y,Z\rangle_{TM} X-\langle X,Z\rangle_{TM} Y. \]
  
  \begin{pr}\label{pr3} Suppose that $(M,\prsm)$ has constant sectional curvature $c$ and put $\varpi=\frac14c(2-ck)$.  Then, for any $X,Y\in\Ga(TM)$ and $F\in\Ga(\mathrm{so}(TM))$,
  	\begin{enumerate}
  		\item 
  		$ R^{\na^A}(X,Y)Z=-2\varpi X\wedge Y(Z)\esp R^{\na^A}(X,Y)F=-2\varpi [X\wedge Y,F] ,$
  		
  		\item $|B|^2(Z+F)=2\varpi^{2}\om^{p}\left(  (n-1)|Z|^2+2(n-2)|F|^2    \right)$, where $|F|^2=-k\tr(F^2)$.
  	\end{enumerate}

  \end{pr}
  
  \begin{proof}\begin{enumerate} \item We have $H_XY=-\frac12R^M(X,Y)=\frac12cX\wedge Y$. Moreover, since the curvature is constant then $\na^M(R^M)=0$ which implies that $\Om$ is also parallel and hence, according to Theorem \ref{main6},
  		$$	R^{\na^A}(X,Y)Z=R^M(X,Y)Z+H_{Y}H_{X}Z-H_{X}H_{Y}Z\esp 
  			R^{\na^A}(X,Y)F=
  			[R^{M}(X,Y),F]+H_{Y}H_{X}F-H_{X}H_{Y}F.$$

  		Now if $(X_i)_{i=1}^n$ is local frame of orthonormal vector fields then
  		\begin{eqnarray*} 
  			\langle H_XF,Y\rangle_{TM}&=&-\frac12k\tr( F\circ R^M(X,Y))
  			=-\frac12ck\sum_{i=1}^n\langle F(X_i),X\wedge Y(X_i)\rangle_{TM} \\
  			&=&-\frac12ck\sum_{i=1}^n\left(\langle Y,X_i\rangle_{TM}\langle F(X_i),X\rangle_{TM}-\langle X,X_i\rangle_{TM}\langle F(X_i),Y\rangle_{TM} \right)\\
  			&=&-ck\langle F(Y),X\rangle_{TM}.
  		\end{eqnarray*}Thus $H_XF=ckF(X)$. So
  		\begin{eqnarray*}
  			\;[H_Y,H_X]Z&=&\frac12(H_YR^M(Z,X)+H_XR^M(Y,Z))\\&=&\frac12ck(R^M(Z,X)Y+R^M(Y,Z)X)\\
  			&=&-\frac12ckR^M(X,Y)Z.
  		\end{eqnarray*}
  		Thus
  		\[ R^{\na^A}(X,Y)Z=\frac12(2-ck) R^M(X,Y)Z=-\frac12c(2-ck)X\wedge Y(Z).\]
  		On the other hand,
  		\begin{eqnarray*}
  			\;[H_Y,H_X]F&=&ck(H_YF(X)-H_XF(Y))\\
  			&=&-\frac12ck(R^M(Y,F(X))+R^M(F(Y),X)),\\
  			&=&-\frac12c^2k([F,X\wedge Y]).
  		\end{eqnarray*}This completes the proof of the first part.

  		\item	We have, for $a=Z+F$,
  		\begin{eqnarray*}
  			h_{p,q}(B_{X^h}Y^h,B_{X^h}Y^h)(a)&=&\frac14h_{p,q}((R^{\na^A}(X,Y)a)^v,(R^{\na^A}(X,Y)a)^v)\\
  			&=&\frac14\om^p\langle R^{\na^A}(X,Y)a,R^{\na^A}(X,Y)a\rangle_A\\
  			&=&\frac14\om^p\left( \langle R^{\na^A}(X,Y)Z,R^{\na^A}(X,Y)Z\rangle_{TM}-k\tr(R^{\na^A}(X,Y)F\circ R^{\na^A}(X,Y)F)     \right).
  		\end{eqnarray*}Let pursue
  		\begin{eqnarray*}
  			\langle R^{\na^A}(X,Y)Z,R^{\na^A}(X,Y)Z\rangle_{TM}&=&4\varpi^2\langle X\wedge Y(Z),X\wedge Y(Z)\rangle_{TM}\\
  			&=&4\varpi^2| \langle Y,Z\rangle_{TM}X-\langle X,Z\rangle_{TM}Y|^2\\
  			&=&4\varpi^2(\langle Y,Z\rangle_{TM}^2|X|^2+\langle X,Z\rangle_{TM}^2|Y|^2-2\langle Y,Z\rangle_{TM}\langle X,Z\rangle_{TM}\langle X,Y\rangle_{TM}).
  		\end{eqnarray*}On the other hand,
  		\begin{eqnarray*}
  			[X\wedge Y,F](X_i)&=&\langle F(X_i),Y\rangle_{TM}X-\langle F(X_i),X\rangle_{TM}Y-\langle Y,X_i\rangle_{TM}F(X)+\langle X,X_i\rangle_{TM}F(Y).
  		\end{eqnarray*}Thus
  		
  		\begin{eqnarray*}
  			|[X\wedge Y,F](X_i)|^2&=&\langle F(X_i),Y\rangle_{TM}^2|X|^2+\langle F(X_i),X\rangle_{TM}^2|Y|^2+\langle Y,X_i\rangle_{TM}^2|F(X)|^2+\langle X,X_i\rangle_{TM}^2|F(Y)|^2\\
  			&&-2\langle F(X_i),Y\rangle_{TM}\langle F(X_i),X\rangle_{TM}\langle X,Y\rangle +2\langle F(X_i),Y\rangle_{TM}\langle X,X_i\rangle \langle X,F(Y)\rangle_{TM}\\
  			&&+2\langle F(X_i),X\rangle_{TM}\langle Y,X_i\rangle_{TM}\langle Y,F(X)\rangle_{TM}-2\langle Y,X_i\rangle_{TM}\langle X,X_i\rangle_{TM}\langle F(X),F(Y)\rangle_{TM}\\
  			\sum|[X\wedge Y,F](X_i)|^2&=&2|F(Y)|^2|X|^2+2|F(X)|^2|Y|^2-4\langle X,Y\rangle_{TM}\langle F(X),F(Y)\rangle_{TM}-4\langle F(X),Y\rangle_{TM}^2\\
  		\end{eqnarray*}So if $|X|=|Y|=1$ and $\langle X,Y\rangle_{TM}=0$ we get
  		\begin{eqnarray*}
  			\varpi^{-2}\om^{-p}h_{p,q}(B_{X^h}Y^h,B_{X^h}Y^h)(a)&=&\langle Y,Z\rangle_{TM}^2+\langle X,Z\rangle_{TM}^2+2k(|F(Y)|^2+|F(X)|^2-2\langle F(X),Y\rangle_{TM}^2).
  		\end{eqnarray*}
  			We have
  		\[ |B|^2=\sum_{i}h_{p,q}(B_{X_i},B_{X_i}). \]
  		Now
  		\begin{eqnarray*}
  			h_{p,q}(B_{X_i},B_{X_i})&=&\sum_{j\not=i}h_{p,q}(B_{X_i}X_j,B_{X_i}X_j)\\
  			&=&\varpi^{2}\om^{p}\sum_{j\not=i}\left(   \langle X_j,Z\rangle_{TM}^2+\langle X_i,Z\rangle_{TM}^2+2k(|F(X_j)|^2+|F(X_i)|^2-2\langle F(X_i),X_j\rangle_{TM}^2)\right)\\
  			&=&\varpi^{2}\om^{p}\left( |Z|^2+(n-2)\langle X_i,Z\rangle_{TM}^2+2k(k^{-1}|F|^2+(n-2) |F(X_i)|^2-2|F(X_i)|^2   \right).
  		\end{eqnarray*}So
  		\[ |B|^2=\varpi^{2}\om^{p}\left(  (2n-2)|Z|^2+2(2n-4)|F|^2    \right).\qedhere \]
  		
  	\end{enumerate}
  	
  \end{proof}
  
  Note that when $c>0$, $AO(M,2/c)$ has a vanishing principal curvature and one can deduce from Corollary \ref{main2} and Corollary \ref{main3} the following result.
  \begin{theo} If $(M,\prsm)$ has constant sectional curvature $c>0$ then:
  	\begin{enumerate}
  		\item $(AO(M,2/c),h_{0,0})$ is locally symmetric with a non-negative sectional curvature and constant scalar curvature $n(n-1)c$.
  		\item $(AO(M,2/c),h_{2,0})$ is locally symmetric with a non-negative sectional curvature and constant scalar curvature $n(n-1)c+4r(r-1)$. Moreover, if $c=\frac{4(r-1)}{n-1}$ then $(AO(M,2/c),h_{2,0})$ is Einstein with the Einstein constant $4(r-1)$.
  	\end{enumerate}
  	
  \end{theo}
  
 We end this work by giving a proof of Theorem \ref{main5}.
  \begin{proof} The scalar curvature of $h_{1,1}$ on $AO(M,k)$ is given by
  	\[ s^A=s^M+s^v-|B|^2, \]where $s^M$ is the scalar curvature of $M$, $s^v$ the scalar curvature of the fiber and $B$ the O'Neill shape tensor. By using Proposition \ref{f} and \ref{pr3}, we get
  	\[ s^A=n(n-1)c+\frac{(r-1)}{(1+t)^3}\left( (r-2)t^3+4(r-2)t^2+6(r-1)t+3r\right)-
  	\frac{2\varpi^{2}}{1+t}\left(  (n-1)|Z|^2+2(n-2)|F|^2    \right), t=|Z|^2+|F|^2, \]where $r=\frac{n(n+1)}2$ is the rank of the vector bundle $AO(M,k)$.
  	Or
  	\[ \frac{(r-1)}{(1+t)^3}\left( (r-2)t^3+4(r-2)t^2+6(r-1)t+3r\right)=\frac{r-1}{\al^2}(6+(r-2)(\al^2+\al+1)),\al=1+t. \]
  	So
  	\[ s^A= n(n-1)c+\frac{r-1}{\al^2}(6+(r-2)(\al^2+\al+1))
  	-
  	\frac{2\varpi^{2}}{\al}\left(  (n-1)|Z|^2+2(n-2)|F|^2    \right).\]
  	So if $c=0$ then $s^A>0$. So we suppose $c\not=0$.
  	
  	If $n=2$ then $r=3$ and by writing $|Z|^2=\al-1-|F|^2$ we get
  	\[ \al^2s^A=2(c+1-\varpi^{2})\al^2+2(1+\varpi^{2})\al+14+2\varpi^{2}\al|F|^2. \]
  	If $n\geq3$ then by writing $|F|^2=\al-1-|Z|^2$ we get
  	\[ \al^2s^A=(n(n-1)c+(r-1)(r-2)-4\varpi^{2}(n-2))\al^2+((r-1)(r-2)+4\varpi^{2}(n-2))\al+(r-1)(r+4)+
  	2(n-3)\varpi^{2}\al|Z|^2. \]
  	{\bf The case $n=2$.}
  	In this case $s^A>0$ if and only if $c+1-\varpi^{2}\geq0$. Since $\varpi=\frac14c(2-ck)$, this is equivalent to
  	\[ k^2c^4-4kc^3+4(c^2-4c-4)\leq 0. \]
  	This  is equivalent to  $\De=4c^6-4(c^6-4c^5-4c^4)=16c^4(1+c)\geq0$ and
  	\[ \frac{2(c-2\sqrt{1+c})}{c^2}\leq k\leq \frac{2(c+2\sqrt{1+c})}{c^2}. \]
  	Since $k>0$ then we must have $c+2\sqrt{1+c}>0$. If $c>0$ this it is true. If $c<0$ then $c+2\sqrt{1+c}>0$ iff
  	\[ c^2-4c-4<0. \]This equivalent to $c>2(1-\sqrt{2})$.
  	
  	{\bf The case $n\geq3$.}
  	In this case $s^A>0$ if and only if $(n(n-1)c+(r-1)(r-2)-4\varpi^{2}(n-2)\geq0$. Since $\varpi=\frac14c(2-ck)$, and if we put $a=n(n-1)$, $b=(r-1)(r-2)$ and $d=4(n-2)$ this is equivalent to
  	\[ k^2dc^4-4kdc^3+4(dc^2-4ac-4b)\leq 0. \]
  	This  is equivalent to  $\De=4d^2c^6-4(d^2c^6-4dac^5-4dbc^4)=16c^4d(b+ac)\geq0$ and
  	\[ \frac{2(cd-2\sqrt{d}\sqrt{b+ac})}{dc^2}\leq k\leq \frac{2(cd+2\sqrt{d}\sqrt{b+ac})}{dc^2}. \]
  	Since $k>0$ then $cd+2\sqrt{d}\sqrt{b+ac}>0$. if $c>0$ then it is true. If $c<0$ this is equivalent to 
  	\[ c^2d-4ac-4b<0. \]This is equivalent to
  	\[ c>\frac{2(a-\sqrt{a^2+bd})}{d}\esp c>-\frac{b}{a}. \]
  	Or $-\frac{b}{a}<\frac{2(a-\sqrt{a^2+bd})}{d}$ so  $s^A>0$ if and only if
  	\[ c>\frac{2(a-\sqrt{a^2+bd})}{d}\esp 0<k\leq \frac{2(cd+2\sqrt{d}\sqrt{b+ac})}{dc^2}. \qedhere\]
  \end{proof}

 \end{document}